\documentclass{amsart}
\usepackage[dvips]{color}

\usepackage{graphicx}
\usepackage{epsfig}
\usepackage{tikz}

\usepackage{latexsym}
\usepackage{amssymb}
\usepackage{amsmath}
\usepackage{amsthm}
\usepackage{amscd}

\usepackage{enumerate}

\theoremstyle{plain}
\newtheorem{thm}{Theorem}[section]
\newtheorem{lem}[thm]{Lemma}
\newtheorem{defin}[thm]{Definition}
\newtheorem{prop}[thm]{Proposition}

\theoremstyle{definition}

\newtheorem{cor}[thm]{Corollary}
\begin{document}

\title[Thick Limit Sets]{Path-connectivity of Thick Laminations, and Markov Processes
  with Thick Limit Sets} \author{Jon Chaika} \author{Sebastian Hensel}
\maketitle


\section{Introduction}
Let $S$ be a closed surface, and $\mathcal{PML}(S)$ the sphere of
projective measured laminations (or foliations). In this article, we
are concerned with paths (and path-connected subsets) consisting of
dynamically and geometrically interesting laminations.

The core property we are interested in is \emph{thickness}, a property
motivated from Teichm\"uller theory. Suppose that $X_0$ is a basepoint
in Teichm\"uller space, which we interpret as the moduli space of
marked hyperbolic surfaces. A lamination $\lambda$ is called
\emph{cobounded}, if the Teichm\"uller ray $\rho$ with initial point
$X_0$ and vertical lamination $\lambda$ projects into a compact
subset of moduli space. A classical theorem of Mumford states that the
set of all surfaces whose systole (i.e. the shortest closed geodesic)
is not shorter than $\epsilon$ is compact in moduli space. Using this
natural exhaustion by compact sets, one can quantify coboundedness: we
say that the ray $\rho$ is \emph{$\epsilon$--thick} if the systole is
not shorter than $\epsilon$ for each point on the ray. We then say
that $\lambda$ is \emph{$\epsilon$--thick relative to $X_0$}. 

Any cobounded lamination is $\delta$--thick for
some $\delta$ relative to our chosen basepoint $X_0$. Note however, that coboundedness is a
mapping-class-group-invariant notion, and does not depend on the
basepoint, while thickness does delicately depend on the basepoint. In
particular, the set of $\epsilon$--thick lamination is closed, while
the set of cobounded lamination is dense in $\mathcal{PML}(S)$. There
can also be paths of cobounded laminations which are not
$\delta$--thick for any $\delta$.

In previous work \cite{Paths} we have shown that the subsets of
$\mathcal{PML}(S)$ of uniquely ergodic laminations and of cobounded
laminations are path-connected. However, by the nature of the
construction employed in \cite{Paths} we do not expect the paths
guaranteed by that previous theorem to be $\delta$--thick for any
$\delta>0$.

As such, the first goal of this article is a more quantitative version
of the connectivity of cobounded laminations: the set of
$\epsilon$--thick laminations is path-connected within the set of
$\delta$--thick laminations. More precisely, we will show: 
\begin{thm}
  Suppose $X_0\in\mathcal{T}(S)$ is  basepoint. For every $\epsilon>0$ there is a $\delta > 0$
  with the following property. If $\lambda, \eta$ are two
  laminations which are $\epsilon$--thick relative to $X_0$, then there
  is a path in $\mathcal{PML}$ connecting $\lambda$ to
  $\eta$, so that any lamination on that path is
  $\delta$--thick relative to $X_0$.
\end{thm}
As an immediate consequence of this theorem and the work of Cordes \cite{Cor}, we
obtain the following.
\begin{cor}
  The Morse boundary of Teichm\"uller space (or the mapping class
  group of $S$) is path-connected.
\end{cor}
In \cite{Paths} we showed that any two points in the Morse boundary
can be connected by a path in $\mathcal{PML}$ where every point is
cobounded, and hence defines a point in the Morse boundary. However,
we don't know that it is a path in the topology for the Morse
boundary. We establish this corollary via the uniform statement
above. The uniform statement also answers \cite[Question 4]{Paths}.

One reason for our interest in cobounded laminations is the connection
to \emph{convex cocompact subgroups} of the mapping class groups, in
particular the famous question if there is a convex cocompact surface
group. If such a group were to exist, its limit set would be a group
invariant circle consisting only of cobounded laminations.

It is known that there are hyperbolic planes, which
quasi-isometrically embed into thick parts of Teichmüller space
\cite{LS-thick}, and whose (Gromov) boundaries therefore yield 
circles of cobounded laminations. However, these are not invariant
under any group. On the other hand, any known examples of convex
cocompact subgroups are virtually free, and therefore have totally
disconnected limit set.

While we cannot construct a group-invariant cobounded circle, we can
find a path-connected subset which is invariant under many mapping
classes. Namely, we show that there is a shift of finite type in the
mapping class where the limit of the orbit map is a path connected set
of cobounded laminations. More precisely:
\begin{thm}\label{thm:markov}
  There is a finite subset of the mapping class group
  $H=\{h_1,...,h_k\}$ and a finite set of ordered pairs
  $G\subset \{1,...,k\}^2$ so that if one considers
  $\hat{H}=\{h_{n_1}...h_{n_j}:(n_{i},n_{i+1})\in G \text{ for all
  }i<j\}$ then for every $\gamma_0$ in the curve graph, the limit set of $\hat{H}\gamma_0$ in the Gromov boundary of the curve graph is path connected and contained in the cobounded laminations. In fact, if we consider the orbit $HX_0$ in Teichm\"uller
  space, then the limit set $\Lambda$ in $\mathcal{PML}$ (seen as the 
  Thurston compactification) has the property that there is an
  $\epsilon>0$ so that $geod(\lambda,\lambda')$ is $\epsilon$-thick
  for all $\lambda, \lambda' \in \Lambda$.
\end{thm}
The above statement is technical, but it improves our understanding of paths of cobounded laminations. Firstly, prior to this work, paths of cobounded laminations were constructed via point pushes and agreed with point pushes on an open set. Our paths are limits of paths with this property, but do not themselves have this property. Additionally, our paths have a \emph{fractal} structure like the paths coming from a putative convex cocompact subgroup of the mapping class group would, where the image of a path under mapping class group elements give other pieces. Lastly, we were motivated by a specific argument that rules out certain types of path connected cobounded subsets of laminations, which we provide in the appendix. This argument combines dynamics and geometry and works just as well for a limit set coming from a shift of finite type as it does for the limit set of a group. Thus Theorem \ref{thm:markov} suggests pessimism for this argument proving there is no convex cocompact subgroup of the mapping class group.

\subsection{Outline of proof}
In both of our theorems, we use the curve graph to certify
thickness. Namely, the laminations will be constructed as limits of
images of curves $\Psi_i\alpha$ under sequences $(\Psi_i)$ of mapping
classes. We will choose these mapping classes in such a way that the
sequence $\Psi_i\alpha$ is a (parametrized) quasi-geodesic in the
curve graph and in the mapping class group -- which is known to imply
$\delta$--thickness.

Hence, the task becomes to guarantee that the sequences $(\Psi_i)$ can
be chosen flexibly enough to build paths (with prescribed
endpoints). Here again, the Gromov hyperbolicity of the curve graph
enters: if a partial sequence $\Psi_i\alpha, i=1, \ldots, N$ is given,
then the condition that ``$\Psi_{N+1}\alpha$ continues to make
progress'' is a purely local, and fairly mild condition not to
backtrack too much.

We then exploit the fact that we have many paths of cobounded
laminations already (from our previous work in \cite{Paths}) to show
that there are many choices for $\Psi_{N+1}$ from a uniformly finite
set.  To make this sketch precise, we need to carefully study
splitting sequences of train tracks.

\subsection{Outline of paper} In Section 2 we introduce the basic, standard objects of this paper, train track patches on $\mathcal{PML}$, the Curve graph and its Gromov boundary and Teichm\"uller space. In Section 3 we discuss splitting sequences and prove Proposition \ref{prop:divide}, a key result about concatenating mapping classes coming from splitting points in a finite set of paths of uniquely ergodic laminations. Section 4 gives criterion to build paths of thick laminations from limits of concatenations of mapping classes that do not backtrack (too much). Section 5 proves the path connectivity of the Morse boundary. Section 6 proves Theorem \ref{thm:markov}. The paper ends with an appendix showing ruling out that certain types of curves in $\mathcal{PML}$ can be uniformly thick. 

\subsection{Acknowledgements:} The authors would like to thank
Alessandro Sisto for pointing out the application to Morse boundaries,
and interesting discussions. JC was supported in part by NSF grants
DMS-2055354 and DMS-2350393. SH was supported in part by the DFG SPP
2026 ``Geometry at Infinity''.

\section{Tools}
In the following subsections we collect results on the three main
tools we use to build our paths: train tracks (giving charts for
$\mathcal{PML}$ and, later, splitting sequences), curve graphs
(allowing flexible constructions via coarse geometry), and
Teichm\"uller spaces (which we use to certify the coboundedness). The
results here are a combination of standard results, and slightly
original constructions suitable for our purposes.

\subsection{Train Tracks}
In this section we collect some results on train tracks which we will
need throughout to obtain suitable charts for the sphere of projective
measured laminations.

The starting point are so-called \emph{standard train tracks}
discussed in detail in \cite{PH}. By the discussion in \cite[Section 2.6]{PH} (in particular, \cite[Proposition~2.6.2]{PH}), there is a finite
collection of maximal tracks $\tau_1, \ldots, \tau_R$ with the
following properties.
\begin{enumerate}
\item The polyhedra of measures $P(\tau_i)$ cover all of
  $\mathcal{PML}$: $\bigcup_i P(\tau_i) = \mathcal{PML}$.
\item If two such polyhedra intersect, say
  $P(\tau_i) \cap P(\tau_j) \neq \emptyset$, then the tracks
  $\tau_i, \tau_j$ have a common subtrack $\eta_{i,j}$, and we have
  $P(\tau_i) \cap P(\tau_j) = P(\eta_{i,j})$.
\end{enumerate}

In other words, the polyhedra $P(\tau_i)$ define a partition of
$\mathcal{PML}$ into (closed) train track patches.  We call the
$\tau_i$ and all of their subtracks \emph{model tracks} and denote
the set containing these train tracks by $\mathcal{T}$.

If $\tau$ is a train track, then we say that $\sigma$ is obtained by a
\emph{full split} if $\sigma$ is obtained by splitting each large
branch of $\tau$ exactly once. Observe that the polyhedra of the two
possible splits intersect in a codimension-$1$--face and cover all of
the the polyhedron of $\tau$. We call any train track obtained from a
model track by performing $k$ full splits towards a (uniquely ergodic)
lamination a \emph{derived track (of level $k$)}. By construction and
property (2) above, the polyhedra of two derived tracks of level $k$
are disjoint or intersect in a face. Since there are only finitely
many possible types of train tracks up to the action of the mapping
class group, we may (by possibly successively replacing model tracks
by their splits a finite number of time) further assume:
\begin{enumerate}
\item[(3)] Any derived track is the mapping class group image of some
  model track: i.e. if $\sigma$ is obtained from a model track by some
  number of full splits, then $\sigma = \Phi\tau_i$ for a model track
  $\tau_i$ and a mapping class $\Phi$.
\end{enumerate}

The following is an immediate consequence of (3) and the fact that the 
stabiliser of a large track is finite.
\begin{lem}\label{lem:finitely-many-models}
  If $\eta$ is any derived track, then there is a model track $\tau$
  and a mapping class $F$ so that $F\tau = \eta$. In particular, $F$
  induces a linear map between the polyhedra $P(\tau)$ and
  $P(\eta)$. Furthermore, the mapping class $F$ is unique up to finite
  ambiguity.
\end{lem}

For any train track $\tau$ (not necessarily one of the $\tau_i$
above), we call pair of faces of $P(\tau)$ \emph{connectable} if each
of the faces contains a cobounded lamination, and thus there is a path
of cobounded laminations connecting the two faces\footnote{Note that we
  do not assume that this path is completely contained in $P(\tau)$.}.

We also have the following:
\begin{lem}\label{lem:left-to-right-set}
  Let $S$ be the set of all mapping classes $\theta$ so that there
  exist $\eta_i$, $\eta_j$, connectable faces of model tracks so that
  $\theta(\eta_i) = \eta_j$. 
  Then $S$ is finite.
\end{lem}
\begin{proof}
  Since the set of model tracks is finite, it suffices to observe that
  any codimension-1 face of a maximal track is still large, and thus
  has finite stabiliser in the mapping class group.
\end{proof}

We require the following lemma, which will guarantee the existence of
many uniquely ergodic laminations carried by codimension-$1$ faces of
the train tracks in our decomposition. We expect that the result is
known to experts (but we were unable to locate it in the literature), so we give a full proof. 
\begin{lem}\label{lem:many-ues}
  Suppose that $\eta$ is a train track which defines a face of the
  polyhedral decomposition defined by the model tracks $\tau_i$.

  Suppose that $\eta$ carries a uniquely ergodic lamination. Then
  $\eta$ carries infinitely many uniquely ergodic laminations.
\end{lem}
\begin{proof}
  For simplicity of notation, we assume that the uniquely ergodic
  lamination $\lambda$ carried by $\eta$ lies in the interior of
  $P(\eta)$, and furthermore that
  \[ P(\tau^+) \cap P(\tau^-) = P(\eta), \] where $\tau^+ \neq \tau^-$
  are two maximal train tracks from the polyhedral decomposition.  In
  the general case, one would first replace $\eta$ with the subtrack
  $\eta'$ which carries $\lambda$ in the interior of its polyhedron,
  and then consider the finite set $\tau^1, \ldots, \tau^l$ of maximal
  train tracks in the chosen decomposition which have $\eta'$ as a
  subtrack. The argument then proceeds as in the simpler case, with
  the only difference that one has to keep track of $l$ instead of $2$
  maximal tracks. 
 
  \smallskip Let $\eta_n$ be the full splitting sequence of $\eta$
  defined by a uniquely ergodic lamination $\lambda$. We can choose
  sequences $\tau_i^+, \tau_i^-$ (by performing several splits at
  once) so that
  \[ P(\tau^+_n) \cap P(\tau^-_n) = P(\eta_n). \]

  By finiteness of the set of train tracks up to the mapping class
  group, there are indices $n_i$ so that the triples
  $(\tau^+_{n_i}, \tau^-_{n_i}, \eta_{n_i})$ are mapping class group
  images of each other.

  Since $\lambda$ is uniquely ergodic, and the intersection of the
  polyhedra $P(\eta_{n_i})$ therefore consists just of $\lambda$
  \cite{Mosher}, there is a pair of indices $m>n$ and
  a mapping class $f$ so that
  \[ f((\tau^+_n, \tau^-_n, \eta_n)) = (\tau^+_m, \tau^-_m, \eta_m), \]
  and so that the carrying matrix for the $f(\eta_n) < \eta_m$ is positive
  (since $\lambda$ lies in the interior of $P(\eta)$).

  Now, consider the splitting sequence
  \[ \eta_n, \ldots, \eta_m = f(\eta_n), f(\eta_{n+1}), \ldots,
    f(\eta_m)=f^2(\eta_n), \ldots \] obtained by concatenating the
  splitting sequence from $\eta_n$ to $\eta_m$ images under powers of
  $f$. Since the carrying matrix for $f(\eta_n) < \eta_m$ is positive,
  the intersection of the polyhedra $f^k(\eta_n)$ is a single
  lamination, say $\lambda_f$, which is uniquely ergodic.

  This implies that $f$ is pseudo-Anosov with stable lamination
  $\lambda_f$.  In particular its stable lamination $\lambda_f$ is
  carried by $\eta_n$ (hence by $\eta$), and its unstable lamination
  $\mu_f$ is not carried by $\eta_n$ (as otherwise, the iterates under
  $f$ would lie in $f^k(P(\eta_n))$ and therefore converge to
  $\lambda_f$). Furthermore, $f$ sends $P(\eta_n)$ into itself.

  If $\lambda_f \neq \lambda$, then $f^k(\lambda)$ are all distinct
  and contained in $\eta_n$, and the lemma follows. Otherwise, we
  argue as follows. Pick points
  $x \in P(\tau_n^-), y \in P(\tau_n^+)$, and (using
  Theorem~\ref{thm:previous-paths}) a path $c$ of uniquely ergodic
  laminations joining them, which is disjoint from $\lambda_f,
  \mu_f$. By construction, the points $x_n =f^n(x), y_n=f^n(y)$ are
  contained in $P(\tau^-)\cup P(\tau^+)$, lie on different sides of
  $P(\eta)$, and converge to $\lambda_f$. Thus, the paths $f^n(c)$ are
  paths that join $x_n$ to $y_n$ and, by north-south dynamics, will
  eventually be completely contained in a small neighbourhood of
  $\lambda_f$. This implies that for large $n$, the path $f^n(c)$
  intersects $P(\eta_n)$. The intersection points are not
  $\lambda_f,\eta_f$ as those are fixed by $f$ and disjoint from
  $c$. Thus, these intersection points give rise to further uniquely
  ergodic laminations on $P(\eta)$, showing the lemma.
\end{proof}

\subsection{Curve Graphs}
In this section we collect results on the coarse geometry of curve
graphs. These will be used throughout to construct laminations and
certify that they are thick. 

\smallskip First, we need the following local control. In its
definition, we say that a concatenation $a\ast b$ of two
quasigeodesics (where the endpoint of $a$ is the initial point of $b$)
has \emph{$R$--backtracking of length $l$} if there is a terminal
segment of $a$ and an initial segment of $b$ of length $l$ which
$R$--fellowtravel. If we do not specify $R$, we assume $R = 2\delta$
for $\delta$ the hyperbolicity constant of the curve graph. By
hyperbolicity, we have the following standard result (see
e.g. \cite[Lemma~4.2]{Min05} for a much stronger version).

\begin{lem}\label{lem:bdd-backtrack-implies-qgeod}
  For any backtracking allowance $l>0$, there is a length certificate
  $L>0$ and a quasigeodesic quality $K>0$ with the following property.

  Suppose that $g_i$ are geodesics of length at least $L$, so that
  $g_i\ast g_{i+1}$ does not have backtracking of length $l$ for all
  $i$. Then the concatenation
  \[ g_1 \ast \cdots g_n \] is a $K$--quasigeodesic.
\end{lem}

In light of the previous lemma, we introduce the following notation. A
path in the curve graph is an \emph{$l$--backtracking
  broken geodesic} if it is a concatenation
\[ g_1 \ast \cdots g_n \] of geodesics, so that that $g_i\ast g_{i+1}$
does not have backtracking of length $l$ for all $i$, and all $g_i$
have length at least $L$ (for the constant of the previous lemma). We
observe two easy important properties: 
\begin{itemize}
\item Any $l$--backtracking broken geodesic is a
  $K$--quasigeodesic (for the $K$ from Lemma~\ref{lem:bdd-backtrack-implies-qgeod}).
\item If $g_1 \ast \cdots g_n$ is an $\ell$--backtracking broken geodesic,
  $g_{n+1}$ is a geodesic of length at least $L$ whose initial point
  is the endpoint of $g_n$, and if $g_n\ast g_{n+1}$ does not have
  backtracking of length $l$, then $g_1 \ast \cdots g_n\ast g_{n+1}$
  is an $\ell$--backtracking broken geodesic as well.
\end{itemize}

Next, we aim to show that paths of bounded backtracking can easily be
constructed, as long as there is some amount of choice available.

To do so, we first recall a basic result on the Gromov boundary of any
hyperbolic metric space; compare e.g. \cite[Lemma
III.H.3.21]{BH}\footnote{In the source, Bridson-Haefliger write that
  the visual metric induces the same topology for proper spaces -- in
  the case of the non-proper curve graph the topology is defined via
  Gromov products in the first place, so the fact that the metric
  induces the topology is clear.}. Recall that the \emph{Gromov
  product} of two points in the curve graph (relative to a basepoint)
is defined as
\[ (\alpha\cdot\beta)_\gamma = \frac{1}{2}( d(\alpha, \gamma) + d(\beta, \gamma) - d(\alpha, \beta) ). \]
It is a standard fact that the Gromov product extends to the Gromov boundary, and can then be (essentially) used to metrize the Gromov boundary:
\begin{lem}\label{lem:visual-metrics}
  There are constants $\kappa\in[0,1], \epsilon>0$ so that the following
  holds. There is a metric $d_\infty$ on the Gromov boundary
  $\partial_\infty\mathcal{C}(S)$ (inducing the topology) so that for
  any two points $\xi, \xi' \in \partial_{\infty}\mathcal{C}(S)$ we have:
  \[ \kappa e^{-\epsilon(\xi\cdot\xi')_{\gamma_0}} \leq d_\infty(\xi,\xi')
    \leq e^{-\epsilon(\xi\cdot\xi')_{\gamma_0}}. \]
\end{lem}

From now on, we will choose a visual metric on
$\partial_\infty\mathcal{C}(S)$ once and for all. The following is immediate from compactness of paths.
\begin{cor}\label{cor:separating-paths}
  Suppose that $\mathfrak{p}_1, \ldots, \mathfrak{p}_k$ is a finite
  collection of paths in
  $\partial_\infty\mathcal{C}(S)$. Then there is a constant $K$ so
  that the following holds: if $x \in \mathfrak{p}_i, y \in \mathfrak{p}_j$ are two points with
  \[ (x\cdot y)_{\gamma_0} > K, \]
  then $\mathfrak{p}_i, \mathfrak{p}_j$ intersect.
\end{cor}

\begin{lem}\label{lem:backtrack-implies-intersection}
  Suppose that $\mathfrak{P}=\{\mathfrak{p}_i\}$ is a finite set of
  paths in the Gromov boundary.
  Then there is a 
  Gromov product certificate $B > 0$ and a no-backtracking guarantee
  $l>0$ so that the following holds.

  Suppose that $g$ is a curve graph geodesic of length at least $L$ with endpoint
  $\gamma_0$, suppose that there are $t_i \in [0,1], i=1,2$ and curves $\gamma_i$ so that
  \[ (\gamma_i\cdot \mathfrak{p}_i(t_i))_{\gamma_0} > B.\]
  Then, if $g_i$ are curve graph geodesics joining $\gamma_0$ to $\gamma_i$
  with the property that
  \[ g \ast g_i' \] does have backtracking of length $l$ for both $i$,
  then
  \[ \mathfrak{p}_1 \cap \mathfrak{p}_2 \neq \emptyset. \]
\end{lem}
\begin{proof}
  Let $K$ be the bound from Corollary~\ref{cor:separating-paths}
  applies to the set of paths $\mathfrak{P}$.

  Assume that $\mathfrak{p}_i,\mathfrak{p}_j \in \mathfrak{P}$ are disjoint, and therefore
  \[ (x\cdot y)_{\gamma_0} \leq K \] for all
  $x \in \mathfrak{p}_i, y \in \mathfrak{p}_j$ by
  Corollary~\ref{cor:separating-paths}.

  If 
  \[ g \ast g_i', \quad \quad g \ast g_j' \] have backtracking of
  length $l$, then by definition initial segments of length $l$ of $g_i', g_j'$ are
  $2\delta$--fellow-traveling a terminal segment of $g$. This implies that
  these initial segments $4\delta$--fellow-travel each other. Hence, we have that
  \[ (\gamma_i\cdot \gamma_j)_{\gamma_0} \geq l-4\delta. \]
  Suppose that
  \[ (\gamma_i\cdot \mathfrak{p}_i(t_i))_{\gamma_0} > B\]
  holds. Recall (e.g. \cite[Chapter III.H, Remark 3.17~(4)]{BH}) that the Gromov product (on a $\delta$-hyperbolic space) satisfies the following version of the triangle inequality:
    \[ (x\cdot y)_{\gamma_0} \geq \min\{(x\cdot z)_{\gamma_0}, (z\cdot y)_{\gamma_0} \} - 2\delta. \]
    for all $x,y,z$. We thus get
    \[ (\mathfrak{p}_i(t_i) \cdot \mathfrak{p}_j(t_j))_{\gamma_0} \geq
      \min\{(\mathfrak{p}_i(t_i)\cdot \gamma_i)_{\gamma_0},
      (\gamma_i\cdot \mathfrak{p}_j(t_j))_{\gamma_0} \} - 2\delta. \]
    The first entry in the $\min$ is at least $B$ by assumption; to
    estimate the second, we use the property again to see that
    \[ (\gamma_i \cdot \mathfrak{p}_j(t_j))_{\gamma_0} \geq
      \min\{(\gamma_i\cdot \gamma_j)_{\gamma_0}, (\gamma_j\cdot
      \mathfrak{p}_j(t_j))_{\gamma_0} \} - 2\delta. \] Here, the first
    entry in the $\min$ is $l-4\delta$ by the above, while the second
    is at least $B$ by assumption. Putting everything together, we see that
    
    \[ (\mathfrak{p}_i(t_i) \cdot \mathfrak{p}_j(t_j))_{\gamma_0} \geq \min\{ B-2\delta, \ell - 8\delta, B - 4\delta\}. \]
    Hence, if $B > K + 4\delta, \ell > K + 8\delta$, this is a contradiction.

  %
  %
  %
  %
  %
  %
\end{proof}

\subsection{Teichm\"uller Spaces}
We will be interested in constructing thick laminations. As
briefly discussed in the introduction, one needs to be a bit careful
about basepoints when discussing this notion. To be clear, we use the
following definitions.
\begin{defin}\label{defin:thick-cobounded}
  \begin{enumerate}[a)]
  \item Given a basepoint $X_0$ in Teichm\"uller space, we say that a
    ray $\rho$ starting in $X_0$ is \emph{$\epsilon$--thick} if the
    systole on $\rho(t)$ has length at least $\epsilon$ for all $t$.

    We then also say that the lamination defining $\rho$ is
    $\epsilon$--thick (relative to $X_0$).
  \item We say that a ray $\rho$ in Teichm\"uller space is
    \emph{($\delta$--)cobounded}, if some terminal segment
    $\rho:[L,\infty)\to \mathcal{T}$ is $\delta$--thick.
    
    Similarly, we say that the defining lamination is cobounded.
  \item A path $\mathfrak{p}$ in $\mathcal{PML}$ is
      \emph{$\epsilon$}-thick if the rays starting in $X_0$ with
      direction $x$ are $\epsilon$--thick for all
      $x \in \mathfrak{p}$.  
  \end{enumerate}
\end{defin}
To control thickness, we use the following result.
\begin{thm}[{\cite[Theorem~2.1~(1)]{Ham}}]\label{thm:quasiconvex-qgeods}
  For every $\nu>1$ there is a constant $\epsilon=\epsilon(\nu)>0$
  with the following property. Suppose that $J \subset \mathbb{R}$ is
  an interval of diameter at least $1/\epsilon$, and
  $\gamma:J\to \mathcal{T}$ is a $\nu$--quasigeodesic. If the map
  which assigns to $t$ a systole on $\gamma(t)$ is a
  $\nu$--quasigeodesic in the curve graph, then any Teichmüller
  geodesics between points on $\gamma$ lies in an
  $1/\epsilon$--neighbourhood of $\gamma$.
\end{thm}

With this, we can give our main thickness criterion.

\begin{lem}\label{lem:coboundedness-criterion}
  Suppose that $\phi_i, i=1, \ldots, N$ are mapping class group
  elements, and $g_i, i=1, \ldots, N$ curve graph geodesics starting
  in $\gamma_0$. Let $X_0$ be a point in Teichm\"uller space whose systole is
  $\gamma_0$. 

  Then there is a constant $\epsilon>0$ with the
  following property.

  \smallskip Suppose that
  \[ G = g_{m_0} \ast \phi_{n_1} g_{m_1} \ast \phi_{n_1}\phi_{n_2} g_{m_2}
    \ast \ldots \] is a $l$--backtracking (infinite) broken
  geodesic.

  Then the Teichm\"uller geodesic joining $X_0$ to
  $\phi_{n_1}\cdots \phi_{n_k}X_0$ is $\epsilon$--thick for all $k$,
  and thus converges to an $\epsilon$--thick limit ray $\rho$.

  Finally, if $\rho, \rho'$ are two rays of this type starting in
  $X_0$, then any Teichmüller geodesic between $\rho(t), \rho'(s)$ is
  also $\epsilon$--thick.
\end{lem}
\begin{proof}
  Consider the path $\gamma$ in Teichmüller space formed by
  concatenating Teichmüller geodesics segments connecting
  \[ X_0, \phi_{n_1}X_0, \ldots, \phi_{n_1}\phi_{n_2}X_0, \ldots \] By
  assumption, this is a $\nu$--quasigeodesic in Teichmüller space
  whose systoles (which are uniformly close to the broken geodesic
  $G$) form a $\nu$--quasigeodesic in the curve graph, for some $\nu$
  depending only on the inital mapping class group elements and
  $\ell$.

  Hence, Theorem~\ref{thm:quasiconvex-qgeods} applies and shows that a
  Teichmüller geodesic joining the endpoints of $\gamma$, showing that
  it is contained in an $\epsilon$--thick part (with $\epsilon$
  depending only on $\nu$).

  The last statement follows by appealing to Theorem~E of
  \cite{Rafi-hyperbolic} (or from Theorem~\ref{thm:quasiconvex-qgeods} again,
  by considering the path following $\rho(t)$ until the first point
  $\rho(t')$ which is close -- in Teichmüller space and the curve
  graph -- to some point $\rho'(s')$, then following along $\rho'$ to
  $\rho'(s)$).

\end{proof}


\section{Splitting Packages}
\label{sec:splitting-packages}
Recall from the first section that we have chosen a finite set of
model tracks $\tau_1, \ldots, \tau_R$, and that $S$ is the (finite!)
set of mapping classes which relate their codimension--$1$--faces.

We now describe how to associate mapping classes to paths using these
train tracks and splits. The definitions and conventions here are
crucial for all of our later constructions.

\smallskip
Namely, suppose that we are given
a path $\mathfrak{p}$ of uniquely ergodic laminations, and a number
$N>0$. We then obtain a \emph{splitting package (of depth $N$)} consisting of:
\begin{enumerate}
\item a decomposition of $\mathfrak{p}$ into subpaths $\mathfrak{p}_i$
  with endpoints $x_i, x_{i+1}$,
\item splitting sequences $\tau^{(i)}_j, j=0,\ldots,N$ of model
  tracks $\tau^{(i)}_0$ of length $N$,
\item mapping class group elements $\Phi_i$,
\end{enumerate}
so that
\begin{enumerate}[i)]
\item the subpath $\mathfrak{p}_i$ is carried by $\tau^{(i)}_N$, 
\item the endpoints $x_i, x_{i+1}$ lie on the boundary of the
  polyhedron of measures $P(\tau^{(i)}_N)$ of that track,
\item the polyhedra of the tracks $\tau^{(i)}_N, \tau^{(i+1)}_N$ share a face, and
\item $\Phi_i^{-1}\tau^{(i)}_{N}$ is a model track.
\end{enumerate}
If $x$ is a point of the subpath with endpoints $x_i, x_{i+1}$, we say
that $\Phi_i$ is \emph{the mapping class associated to $x$ (in the
  splitting package)}. If $\eta_a, \eta_b$ of $\tau^{(i)}_N$ are the faces 
containing $x_i, x_{i+1}$, then we say that
\[ (\Phi_i, \Phi_i^{-1}\eta_a, \Phi_i^{-1}\eta_b) \] is the \emph{data
  associated to $x$ (in the splitting package)}.

The existence of splitting packages follows since the polyhedra of the
model tracks define a cell structure on $\mathcal{PML}$.  We emphasise
that the splitting package is not unique. Further note that, by
definition of the set $S$, we have that
\[ \Phi_i^{-1}\Phi_{i+1} \in S \]
for the mapping classes in the splitting package of $\mathfrak{p}$, and
\[ \Phi_i^{-1}x_i, \Phi_i^{-1}x_{i+1} \] 
lie in a connectible pair of faces of a model track.

Similarly, if $z \in \mathcal{PML}$ is a minimal lamination, then we can define
  the \emph{splitting package of $z$} to consist of a splitting
  sequence $\tau_j$ of length $N$ and mapping classes $\Phi_i$ so that
  $z$ is carried by $\tau_j$ for all $j$, and $\Phi^{-1}_i\tau_i$ is
  a model track.

  With this convention, the splitting packages associated to paths are
  the splitting packages of the endpoints of the subpaths
  $\mathfrak{p}_i$.

  We now have
  \begin{lem}\label{lem:increment}
    Suppose $z$ is a 
    point in $\mathcal{PML}$, and
    $(\Phi_i, \tau_i), (\Phi'_i, \tau'_i)$ are two splitting packages
    of some depth $k$ for $z$. Let
    \[ x = \Phi_k^{-1}(z) \in P(\Phi_k^{-1}\tau_k), x' =
      (\Phi'_k)^{-1}(z) \in P((\Phi'_k)^{-1}\tau'_k), \] and assume that
    $x,x'$ are contained in faces $\eta, \eta'$ of the corresponding
    model tracks.
    
    Then the mapping class $\theta = (\Phi_k')^{-1}\Phi_k$ is contained in
    $S_\eta$, and in fact
    \[ \theta \eta = \eta', \theta x = x'. \]
  \end{lem}
  \begin{proof}
    This is essentially by definition, noting that $\theta x = x'$ since
    the two splitting packages describe the same point $z$.
  \end{proof}


\smallskip The next proposition
describes crucial properties of the splitting package for a given set
of paths in the Gromov boundary.
\begin{prop} \label{prop:divide} Suppose that $\mathfrak{P}$ is a
  finite set of uniquely ergodic paths in $\mathcal{PML}$ (identified with
  the corresponding paths in the Gromov boundary of the curve
  graph). Further, suppose that a number $B>0$ is given.

  If $N > 0$ is chosen large enough, the following properties
  are true for all mapping classes in the splitting packages of depth $N$ of the
  paths in $\mathfrak{P}$:
  \begin{enumerate}[a)]
  \item If a given mapping class $\Psi$ appears as $\Phi_i, \Phi'_j$
    in the splitting packages for two paths
    $\mathfrak{p}, \mathfrak{p}'$, then the paths share a point. 
  \item Suppose that $\Psi$ is any mapping class so that 
    $d(\gamma_0, \Psi\gamma_0) > L$, and $\Phi_i, \Phi'_j$ are mapping
    classes in the splitting packages of two paths
    $\mathfrak{p}, \mathfrak{p}' \in \mathfrak{P}$. If the
    concatenations
    \[ [\gamma_0, \Psi\gamma_0] \ast [\Psi\gamma_0,
      \Psi\Phi_i\gamma_0], \quad\mbox{and}\quad [\gamma_0,
      \Psi\gamma_0] \ast [\Psi\gamma_0, \Psi\Phi_j'\gamma_0] \] both
    have at least $l$ backtracking, then $\mathfrak{p}, \mathfrak{p}'$
    share a point.
  \end{enumerate}
\end{prop}
\begin{proof}
  First, apply Lemma~\ref{lem:backtrack-implies-intersection} to
  obtain certificates $B,l>0$, and
  Corollary~\ref{cor:separating-paths} to obtain a $K$ from the given
  set of paths.

  \smallskip If $x \in \mathcal{PML}$ is uniquely ergodic, and carried
  by a model train track $\tau_i$, consider a splitting sequence
  $\tau_i(j)$ in the direction of $x$. Since vertex cycles of splitting sequences form
  uniform quasigeodesics in the curve graph, the distance
  \[ d(\gamma(\tau_i(1)), \gamma(\tau_i(j))) > 100\cdot B \] for all
  large enough $j$. Furthermore, the polyhedra $P(\tau_i(j))$ define
  open neighbourhoods of the Gromov boundary point $x$ which intersect
  in $\{x\}$. Thus, for $N$ large enough, we can further guarantee that 
  \[ (y \cdot y')_{\gamma_0} > 2K\]
  for all curves or minimal laminations $y, y'$ carried by $\tau_i(j), j>N$.

  By finiteness of the set $\mathfrak{P}$, and compactness of the
  paths, we can choose a number $N$ which has the above properties for all points $x$
  on all paths $\mathfrak{p}\in\mathfrak{P}$, and we do so.

  Corollary~\ref{cor:separating-paths} then implies
  condition~a). Namely, if the same mapping class $\Psi$ appears in
  the splitting package of points $x_i \in \mathfrak{p}_i$, then
  $\Psi(\gamma_0)$ is a vertex cycle of $N$--splits $\tau'_i$ of
  model tracks in the directions of $x_i$. This implies that the
  Gromov product of $x_1, x_2$ is at least $K$, and thus the Corollary
  applies by the above choice.

  \smallskip Property~b) is immediate from the choices and
  Lemma~\ref{lem:backtrack-implies-intersection}, since by our choices
  above the Gromov product of $\Phi_i\gamma_0$ with a point on
  $\mathfrak{p}$ and the Gromov product of $\Phi_j'\gamma_0$ with a
  point on $\mathfrak{p}'$ is bigger than $B$. 
\end{proof}

\section{Building Thick Paths}
\label{sec:thick-paths}
The goal of this section is to use the machinery developed so far to
show that any two $\epsilon$--thick laminations can be joined by a
path $\mathfrak{p}$ which traces out a $\epsilon'$--thick region in
Teichm\"uller space, proving the first main theorem from the introduction. The precise version is as follows. 
\begin{thm}\label{thm:filling-thick}
  Suppose $\epsilon > 0$ and two $\epsilon$--thick Teichm\"uller rays
  $\rho_0, \rho_1:[0, \infty) \to \mathcal{T}(S)$ which are
  $\epsilon$--thick (i.e. the systole on $\rho_i(t)$ has length at
  least $\epsilon$ for all $i,t$) are given.

  Then there is an $\epsilon'>0$ and a continuous family of paths
  \[ H:[0,1] \times [0, \infty) \to \mathcal{T}(S) \] so that
  $H(0,t) = \rho_0(t), H(1,t) = \rho_1(t), H(s,0) = H(0,0)$ and so that each
  $H(s,\cdot)$ is an $\epsilon'$--thick Teichm\"uller ray.
\end{thm}
This corollary will be used in the next section to show our main
result about Morse boundaries. Although many ideas are similar to our
main result about Markov subshifts, the result is technically much
less demanding. A reader interested only in the latter result may (at
their own risk) skip this section.

\bigskip Throughout, we fix a basepoint $X_0$ in Teichm\"uller space,
so that we are able to talk about thick rays and laminations. We begin
by choosing
\begin{itemize}
\item \textbf{(Decomposition)} A polyhedral decomposition of
  $\mathcal{PML}$ given by train track polyhedra $P(\tau_i)$ as in
  Section~\ref{lem:finitely-many-models}.
\end{itemize}

We now require the following result about connectivity of the set of
uniquely ergodic laminations.
\begin{thm}\label{thm:previous-paths}
  The set of uniquely ergodic laminations is path connected. In
  addition, if $Z$ is a finite sets of uniquely ergodic laminations,
  we can choose embedded paths between any pair $z,z' \in Z$ which
  intersect only if they have an endpoint in common.
\end{thm}

Using these results, we can then choose
\begin{itemize}
\item \textbf{(Connection Points)} Let $\tau$ be either a model track,
  or a single split of a model track, and let $P(\eta) < P(\tau)$ a
  face. If $\eta$ carries any uniquely ergodic laminations, we then
  pick a set of $6$ uniquely ergodic laminations $C(\tau, \eta)$ on
  $P(\eta)$. This is possible by Lemma~\ref{lem:many-ues}.
  
\item \textbf{(Model Paths)} Let $(\tau, \eta), (\tau', \eta')$ be two
  pairs as in (Connection Points).

  Then, for any pair of points
  $x \in C(\tau, \eta), y \in C(\tau', \eta')$ and any
  $\theta \in S_{\eta'}$ we choose a path $\mathfrak{p}(x,\theta y)$
  of uniquely ergodic laminations connecting $x$ to $\theta y$. We
  require \textbf{(Minimal Intersection)}: paths
  $\mathfrak{p}(x,\theta y)$ are disjoint if they do not have an
  endpoint in common.
\end{itemize}
Theorem~\ref{thm:previous-paths} ensures that we can choose these
points and paths. We also mention that (Model Paths) chooses (many) more
paths than we will need -- in this section we try to keep the
assumptions as simple as possible, but later we will have to be more
careful, when we try to build a Markov process.

Next, we need to choose a few constants. Namely, we pick
\begin{itemize}
\item Constants $B, l$ as in
  Lemma~\ref{lem:backtrack-implies-intersection} for the union $M$ of
  all $M(\tau_i, \eta)$,
\item a \emph{splitting depth} $N > 0$, as in Proposition~\ref{prop:divide}
  for $B,l$ for the collection $M$.
\end{itemize}

We record the following lemma.
\begin{lem}
  Suppose that $\tau, \tau'$ are two derived tracks of level $N$
  (i.e. train tracks obtained from model tracks by $N$ splits), which
  carry uniquely ergodic laminations. Then there is a sequence
  \[ \tau = \tau_1, \ldots, \tau_n = \tau' \] of derived tracks of
  level $N$ so that $\tau_i, \tau_{i+1}$ have a subtrack $\eta_i$ in
  common which carries a uniquely ergodic lamination.
\end{lem}
\begin{proof}
  By Theorem~\ref{thm:previous-paths}, there is a path of uniquely
  ergodic laminations joining a lamination $\lambda$ carried by $\tau$
  to $\lambda'$ carried by $\tau'$. Now, we can take the $\tau_i$ to
  be the tracks appearing in a splitting package of that path.
\end{proof} 

Now, suppose that $\lambda_1, \lambda_2$ are $\epsilon$--thick
laminations, and let $\sigma^{(i)}_n$ be $N$--step splitting sequences
of model tracks towards $\lambda_i$, and $\Psi^{(i)}_k$ be mapping classes so that
\[ \left(\Psi^{(i)}_k\right)^{-1}\sigma^{(i)}_k \]
is a model track.

\smallskip We aim to construct a thick path $\mathfrak{p}_\infty$
between $\lambda_1, \lambda_2$ by concatenating mapping class group
images of model paths. Every point $\mathfrak{p}_\infty(t)$ on this
path will be constructed as a limit of mapping class group images of a
basepoint curve $\alpha_0$ of a specific form: 
\begin{itemize}
\item \textbf{(Limit Mapping Class Expansion)} There is a constant $B$
  (independent of $t$), and for any $t\in[0,1]$ there is a a ``small''
  mapping class $\Upsilon(t)$ whose word norm is bounded by $B$, a number
  $N_1(t)$ (which diverges to $\infty$ as $t\to 0, 1$), and a sequence $\Phi_j(t)$ of mapping classes (which depend on $t$), all of which come from splitting packages of model paths, so that: 
    \[
      \mathfrak{p}_\infty(t) = \lim_{M\to \infty}
      \Psi^{(i)}_{N_1(t)}\Upsilon(t)\Phi_{N_1+2}(t)\Phi_{(N_1+3)}(t)\cdots
      \Phi_{N_1+M}(t)\alpha_0,
    \]
    where $i=0$ for $t\leq 1/2$ and $i=1$ for $t > 1/2$.
\end{itemize}
Below, for readability we will usually suppress the dependence of $t$ in (Limit
Mapping Class Expansion) when talking about a specific point on the
path. Since our constants are uniform we hope that this will cause no confusion. To ensure the convergence actually works, we require further
conditions.
\begin{itemize}
\item \textbf{(Limit No-Backtracking)}
  In the description above, the concatenations
  \[ \left[\left(\Psi^{(i)}_{N_1}\Upsilon\right)^{-1}\alpha_0, \alpha_0\right] \ast \left[\alpha_0, \Phi_{N_1+2} \alpha_0\right] \]
  and
  \[ \left[\Phi_i^{-1}\alpha_0, \alpha_0 \right] \ast [\alpha_0, \Phi_{i+1}\alpha_0 ] \]
  have at most $\ell$ backtracking.
\end{itemize}
By Lemma~\ref{lem:coboundedness-criterion}, (Limit No-Backtracking)
immediately implies that such a path satisfies our requirements:
\begin{cor}\label{cor:noback-to-thick}
  There is a number $\delta>0$ so that the Teichmüller geodesic ray
  starting in $X_0$ with endpoint $\mathfrak{p}_\infty(t)$ is
  $\delta$--thick. Furthermore, the Teichmüller geodesics joining
  points on any two such rays are also $\delta$--thick.
\end{cor}

The path $\mathfrak{p}_\infty$ will be the limit of approximating
paths $\mathfrak{p}_n$, which in turn will be constructed
inductively. To find the paths, we will also inductively construct
sequences of derived tracks of level $n$
\[ \tau^{(n)}_i = \Phi^{(n)}_i\overline{\tau}^{(n)}_i , i = 1, \ldots,
  m_n, \] where the $\overline{\tau}^{(n)}_i$ are model tracks, and
the $\Phi^{(n)}_i$ are mapping classes, so that the following hold:

\begin{itemize}
\item \textbf{(Coarse Path)} The train tracks
  $\tau^{(n)}_i, \tau^{(n)}_{i+1}$ share a common subtrack.
  
\item \textbf{(Tracking)} $\lambda_1 \in P(\tau^{(n)}_1), \lambda_2 \in P(\tau^{(n)}_{m_n})$.
  
\item \textbf{(Refinement)} There is a monotonic decomposition
  \[ [1,m_{n+1}] = I_1 \cup \cdots \cup I_{m_n} \]
  into subintervals, so that for all $i \in I_j$ we have one of the following:
  \begin{enumerate}[a)]
  \item if $i = 1$ and $k=1$, or $i = m_{n+1}$ and $k=2$, 
    \[ \Phi^{(n+1)}_i = \Psi^{(k)}_{n+1}. \]
    (i.e. ``the leftmost and rightmost mapping class code like the endpoints'')
  \item if $j=1$ or $j=m_n$, 
    \[ \Phi^{(n+1)}_i = \Phi^{(n)}_j \Upsilon, \] where $\Upsilon$ is
    a mapping class of norm at most $K_1$ (i.e. ``within the first and
    last interval, arbitrary small refinements are allowed'').    
    \item \[ \Phi^{(n+1)}_i = \Phi^{(n)}_j \Psi_i, \] where $\Psi_i$ is
      a mapping class in a splitting package towards of a model path
      (i.e. ``on the inside of the path, increments are splitting packages'').
    \end{enumerate}
  \item \textbf{(No-backtracking)} For any $i \in I_j$ as above,
    $[\alpha_0, \Psi_i]$ has no more than $\ell$ backtracking with
    $[(\Phi^{(n)}_j)^{-1}\alpha_0, \alpha_0]$
\end{itemize}

Before proving that such sequences exist, we note that they are
sufficient for our purposes. We say that a sequence of indices $j_i$
is \emph{admissible} if for all $j_i$, when we let $J_j$ be the 
intervals from (Refinement) at stage $n$:
\[ j_{n+1} \in J_{j_n}. \]
We then have
\begin{cor}\label{cor:admissible-limits}
  For any admissible sequence $(j_i)$, the corresponding infinite mapping class product
  \[ \Phi^{(1)}_{j_1}\Phi^{(2)}_{j_2}\cdots \Phi^{(k)}_{j_k}\cdots \alpha_0 \]
  satisfies (Limit Mapping Class Expansion) and (Limit No-Backtracking), and therefore
  has a well-defined limit in the boundary of the curve graph.

  The limits of the sequences $(1)$ and $(m_n)$ are $\lambda_1$ and $\lambda_2$.
\end{cor}
\begin{proof}
  To prove (Limit Mapping Class Expansion) it suffices to observe that
  by the rules in (Refinement), along an admissible sequence, a)
  occurs for an initial segment, then b) occurs at most once, and then
  c) occurs for all further choices. (Limit-No-Backtracking) is
  immediate from (No-Backtracking).

  The fact that the sequences of curves converge follows by
  Lemma~\ref{lem:coboundedness-criterion}.

  Finally, the last claim follows from case a) in (Refinement).
\end{proof}
Finally, we note that the limits from
Corollary~\ref{cor:admissible-limits} form a path $\mathfrak{p}_\infty$ as desired.
\begin{cor}\label{cor:close}
  For any $\epsilon>0$ there exists an $N$ so that if $(j_i)$ and
  $(j'_i)$ are two admissible sequences which share an initial segment
  of length $N$, then the corresponding limits from
  Corollary~\ref{cor:admissible-limits} have visual distance at most
  $\epsilon$.
\end{cor} 
Under our assumptions, the mapping class images of $\alpha_0$ make progress from $\alpha_0$ uniformly fast in $N$. That is, letting $\hat{\Phi_j^{(n)}}=\prod_\ell \Phi^{(\ell)}_{j_\ell}$, 
$\lim_N \inf_{i} d(\hat{\Phi}_i^{(N)}\alpha_0,\alpha_0)=\infty$.  Then we relate that they diverge far from $\alpha_0$ to the Gromov product.


\begin{proof}
Let $\lambda,$ $\Phi_{j_i}^{(i)}$ be as in Corollary \ref{cor:admissible-limits}. By the limit no back tracking condition, for each $k$ there exists $M$ so that for each $\lambda$ and corresponding $\Phi_{j_1}^{(1)}...\Phi_{j_k}^{(k)}\alpha_0$ 
a $k$-quasigeodesic from $\alpha_0$ to $\lambda$ enters the $M$ neighborhood of $\Phi_{j_1}^{(1)}...\Phi_{j_k}^{(k)}\alpha_0$. 
Similarly, by the no backtracking condition and the length of the splitting package, for every $R$ there exists $k$ so that $d_{\mathcal{C}}(\Phi_{j_1}^{(1)}...\Phi_{j_k}^{(k)}\alpha_0,\alpha_0)>R.$ By the definition of Gromov product we have the corollary. 
\end{proof}

We now begin with the inductive construction of the train track
sequences. To construct these, we will use \emph{guide paths} which
are concatenations of images of model paths. To gain control at the
endpoints, we need the following.
\begin{lem}\label{lem:end-partial-refinement}
  For any $n$ there is a constant $K_1=K_1(n)$ depending on $n$, so that the
  following holds.

  Suppose that $\psi$ is a mapping class, $\tau$ a model track,
  $\eta < \tau$ a subtrack carrying uniquely ergodic laminations, and
  $\tau_n$ a track obtained from $\tau$ by $n$ splits. Suppose
  $\tau_n=f(\tau')$ for a model track $\tau'$, and $\eta' < \tau'$ a
  subtrack carrying uniquely ergodic laminations.

  Then for all but at most two choices of endpoints
  $x \in C(\tau, \eta), y \in C(\tau', \eta')$, there is a path
  \[ \mathfrak{p} = \mathfrak{p}_1 \ast \mathfrak{p}_2 \ast \cdots
    \ast \mathfrak{p}_n \] from $x$ to $f(y)$, where each
  $\mathfrak{p}_i$ is the image of a model path under a mapping class
  of norm at most $K_1$, and furthermore for any mapping class $\phi$ in the
  splitting package of $\mathfrak{p}$, the concatenation
  \[ [\psi^{-1}\alpha_0, \alpha_0] \ast [\alpha_0, \phi\alpha_0] \]
  has at most $\ell$ backtracking.
\end{lem}
\begin{proof}
  First, consider the splitting sequence
  \[ \tau = \tau_1 > \tau_2 > \cdots > \tau_n \] corresponding
  to the assumption that $\tau_n$ is obtained from $\tau$ by $n$
  splits. We choose mapping classes $\phi_i$ so that
  \[ \hat{\tau}_i = \phi_i^{-1}\tau_i \] is a model track. We may
  choose $\phi_1 = \mathrm{id}$ since $\tau_1=\tau$ is itself a model
  track, and $\phi_n = f$ by the assumption. Furthermore, choose
  codimension-$1$ faces $\hat{\eta}_i$ of the $\hat{\tau}_i$ which
  carry uniquely ergodic laminations. We put $\hat{\eta}_1 = \eta$ and
  $\eta' = \eta_n$.

  Observe that there is a constant $K_1$ so that all $\phi_i$ have
  word norm at most $K_1$, since the set of train tracks obtained from
  a model track by at most $n$ splits is finite, and therefore all
  $\phi_i$ lie in some finite set of the mapping class
  group, just determined by $n$ and the choice in (Decomposition).

  The desired path $\mathfrak{p}$ will have the form
  \[ 
    \phi_1 \hat{\mathfrak{p}}_1 \ast \cdots \phi_n \hat{\mathfrak{p}}_n.
  \]
  where each $\hat{\mathfrak{p}}_i$ is a model path, which will be
  chosen below. Being more precise, the model path
  $\hat{\mathfrak{p}}_i$ will join a point $x_i$ to a point $y_i$. To
  satisfy our conditions, we want to require
  \begin{itemize}
  \item $x_i \in C(\hat{\tau}_i, \hat{\eta}_i)$ for all $i$,
  \item $y_n \in C(\hat{\tau}_n, \hat{\eta}_n)$
  \item $\phi_i(y_i) = \phi_{i+1}(x_{i+1})$, so that the concatenation
    is indeed a path, and
  \item $\hat{\mathfrak{p}}_i$ has at most $\ell$ backtracking with
    $\psi\phi_i$.
  \end{itemize}
  As a first step towards showing that these requirements are possible,
  assume first that we pick only initial points $x_i$ for all $i$.
  Then we put
  \[ y_i = \phi_i^{-1}\phi_{i+1}(x_{i+1}) \] and observe that it is contained in a face of
  $\phi_i^{-1}(\tau_{i+1})$. Since $\tau_{i+1}$ is a split of
  $\tau_i$, $\phi_i^{-1}(\tau_{i+1})$ is a split of the model track
  $\hat{\tau}_i = \phi_i^{-1}(\tau_i)$. Thus, by (Model Paths), there
  will be $\hat{\mathfrak{p}}_i$ with the endpoints $x_i, y_i$ for all
  $i$.

  We therefore only need to make sure that we can choose initial
  points $x_i$ so that the no-backtracking condition is satisfied. By
  Lemma~\ref{lem:backtrack-implies-intersection}, there is a set $B_i$
  of at most two points on $C(\tau_i, \eta_i)$ and similarly a set
  $B_i'$ of at most two points on
  $C(\phi_i^{-1}(\tau_{i+1}), \phi_i^{-1}(\eta_{i+1})$ so that
  \emph{any} model path with endpoints distinct from $B_i, B_i'$ has
  at most $\ell$ backtracking with $\psi\phi_i$. 

  Thus, by excluding at most four points in each $C(\tau_i, \eta_i)$
  we can make a choice of $x_i$ (and therefore $y_i)$ and model paths
  $\hat{\mathfrak{p}}_i$ so that all three conditions are satisfied,
  showing the lemma.
\end{proof}

We are now ready to describe the inductive procedure (and therefore,
finish the proof). The construction is very similar to the argument
in the proof of the previous lemma. Namely, suppose we are given
\[ \tau^{(n)}_i = \Phi^{(n)}_i\overline{\tau}^{(n)}_i , i = 1, \ldots,
  m_n, \] with the desired properties. We now aim to construct a path
\[ \mathfrak{p} = \mathfrak{p}_- \ast \mathfrak{p}_1 \ast
  \mathfrak{p}_2 \ast \cdots \ast \mathfrak{p}_n \ast
  \mathfrak{p}_+ \] where $\mathfrak{p}_-, \mathfrak{p}_+$ are images
under $\Phi^{(n)}_1, \Phi^{(n)}_{m_n}$ of a path obtained by applying
Lemma~\ref{lem:end-partial-refinement}, and the other $\mathfrak{p}_i$
are images of model paths under the $\Phi^{(n)}_i$. The new sequence
of train tracks is the obtained from the splitting package of the path
$\mathfrak{p}$. Arguing (with excluded boundary points) as in the
proof of Lemma~\ref{lem:end-partial-refinement}, we can guarantee the
no-backtracking condition. Note that all increment
  mapping classes come from splitting towards a model path, except
  possibly in $\mathfrak{p}_-, \mathfrak{p}_+$. Hence, the choices are compatible with (Refinement) a) through c).

\section{Morse Boundaries}
In this section, we explain how the results about laminations relate
to the Morse boundary of Teichm\"uller space and the mapping class
group. We begin by recalling some concepts (see \cite{Cor},
\cite{ChCorSis} for details). We also remark that \cite{Cor} contains
a technical omission, which is addressed in an appendix which at the time
of writing is only available in the preprint version \cite{ChCorSis}. 

\smallskip Suppose $N:\mathbb{R}^+\times\mathbb{R}^+ \to \mathbb{R}^+$
is a function, which we call a \emph{Morse gauge}. It is called a
\emph{refined Morse gauge} if it is non-decreasing (in both
coordinates), and continuous in the second coordinate. We assume that
any gauge we use is refined. This is not a serious restriction
(compare \cite[Lemma A.4]{ChCorSis}).

We then say that a geodesic $\rho:I \to X$ in a proper geodesic metric
space is \emph{(refined) $N$--Morse} if for any
$(\lambda, \epsilon)$--quasigeodesic with endpoints on $\rho$ lies in
a closed $N(\lambda, \epsilon)$--neighbourhood of $\rho$.
When we omit the mention and simply speak of a Morse geodesic $\rho$,
we mean that there is some refined gauge $N$ so that $\rho$ is
$N$--Morse.

We say that two Morse rays $\rho, \rho'$ in $X$ are equivalent, if
they have bounded Hausdorff distance (in fact, if this holds, then the
Hausdorff distance can be bounded by a constant just depending on the
Morse gauges; see \cite[Key Lemma]{Cor}).

As a set, the Morse boundary $\partial_M X$ of $X$ is the set of
equivalence classes of (refined) Morse rays. Here as well the
restriction to refined rays is no restriction \cite[Theorem A.10]{ChCorSis}.

To define a topology, we need to be a bit careful about
gauges. Namely, choose a basepoint $e \in X$, and fix a refined gauge
$N$. Then let $\partial_M^N X_e$ be the set of $N$--Morse rays
starting in $e$. As a set of maps $[0, \infty) \to X$ this set carries
a natural compact-open topology (uniformly convergence on compact
intervals). The fact that the Morse gauge is fixed (and refined)
implies that $\partial_M^N X_e$ is compact with this topology
\cite[Lemma A.5]{ChCorSis}.

Consider the set $\mathcal{M}$ of all possible refined Morse gauges, and say
$N \leq N'$ (for $N,N' \in \mathcal{M}$) if
$N(\lambda, \epsilon) \leq N'(\lambda, \epsilon)$ for all
$\lambda, \epsilon$. We then define the \emph{Morse boundary} as a
direct limit:
\[ \partial_M X = \varinjlim \partial_M^NX_e \] with the direct limit
topology. Implicit in this notation is that the limit does not depend on the chosen basepoint $e$.
Concretely, the $\partial_M^NX_e$ are all subspaces of the
Morse boundary, and a set $U \subset \partial_M X$ is open exactly if
the intersections $U \cap \partial_M^NX_e$ are open for all choices of
gauge. 

In particular, if $\gamma:[a,b] \to \partial_M^NX_e$ is a continuous
path, then it also defines a continous path in $\partial_M X$.

\bigskip The Morse boundary of Teichm\"uller space and of the
mapping class group are described in \cite{Cor}. The core ingredient is
that it is known that a Teichm\"uller geodesic ray is Morse (for some
gauge) exactly if it is thick (for some thickness constant
$\epsilon$) \cite{Minsky}.

As a consequence of the results in the previous section, we then obtain
\begin{cor}
  The Morse boundary of Teichm\"uller space (and of the mapping class
  group) are path-connected.
\end{cor}
\begin{proof}
  Let $\xi, \eta \in \partial_M\mathcal{T}(S)$ be given. These are
  given by Teichm\"uller geodesic rays $\rho_0, \rho_1$ which are Morse
  for some gauge $N$ (compare Lemma 4.1 of [CD19]). Hence, there is
  some $\epsilon$ so that $\rho_0, \rho_1$ are
  $\epsilon$--thick. Apply Theorem~\ref{thm:filling-thick} to find a
  family $H$ of $\epsilon'$--thick Teichm\"uller geodesics
  interpolating between them. Since $\epsilon'$ is uniform, there is
  some Morse gauge $N'$ so that all $H(s,\cdot)$ are $N'$--Morse
  rays. Continuity of $H$ implies that $s \to H(s, \cdot)$ defines a
  continuous path $[0,1] \to \partial_M^{N'}\mathcal{T}_{H(0,0)}$ joining
  $\xi, \eta$.  By the direct limit topology, this then also defines a
  continuous path joining $\xi, \eta$ in $\partial_M\mathcal{T}$.

  The claim about the Morse boundary of the mapping class group
  follows since the Morse boundaries of Teichm\"uller space and the
  mapping class group are homeomorphic, by \cite[Theorem 4.12]{Cor}.
\end{proof}

\section{A Markov Version}
In this section we construct a Markov process whose limit set is
path-connected and consists of $\delta$-thick laminations for some $\delta$. The general idea is similar to our
construction of paths in the Morse boundary, but we have to be more
careful in our choices to ensure the Markov property. Note in this section we choose some actual numbers for readability and concreteness. Our choices are not optimal.

\smallskip We again begin by choosing
\begin{itemize}
\item \textbf{(Decomposition)} A polyhedral decomposition of
  $\mathcal{PML}$ given by train track polyhedra $P(\tau_i)$ as in
  Section~\ref{lem:finitely-many-models}
\item \textbf{(Faces)} We let
  \[ P \] to be the set of all those codimension-$1$ faces $P(\eta)$
  of polyhedra $P(\tau_i)$ chosen in (Decomposition) which contain
  cobounded laminations. 

\item \textbf{(Connection Points)} For any $\eta \in P$, we choose
  a set $Y_\eta$ of cobounded points on $P(\eta)$ of cardinality 20. 

\item \textbf{(Symmetries)} Given $\eta \in P$, we let
  \[ S_\eta \] be the set of all mapping classes $f$ so that
  $f(\eta) \in P$. Note that $S_\eta$ is finite.

\item \textbf{(Model Paths)} We choose model paths between each
  $\theta Y_{\theta^{-1}\eta_a}$ to $Y_{\eta_b}$ where $\eta_a$ and
  $\eta_b$ are a pair of boundary planes on the same $\tau$ and
  $\theta \in S_{\eta_a}$. Let $M$ be the set of all those model
  paths.

\item \textbf{(Increments)} Let $I$ be the set of mapping classes
  $\phi$ obtained in the splitting packages of all paths in $M$.

\item A basepoint $\alpha_0$ in the curve graph.
  
\item Constants $B, l$ as in
  Lemma~\ref{lem:backtrack-implies-intersection} for the collection $M$.
\item a \emph{splitting depth} $N > 0$, as in Proposition~\ref{prop:divide}
  for $B,l$ for the collection $M$.
  
\end{itemize}

We will use the following crucial property
\begin{defin}
  \begin{enumerate}[i)]
  \item A pair of mapping classes $(\phi, \Psi)$ is \emph{PIP} if the
    concatenation
    \[ [\alpha_0, \Psi\alpha_0] \ast [\Psi\alpha_0,
      \Psi\phi\alpha_0] \] does not have backtracking of length
    $l$. 
  \item A pair $(\mathfrak{p}, \Psi)$ of a path $\mathfrak{p} \in M$
    and a mapping class $\Psi$ is \emph{PIP} if $(\phi, \Psi)$ is PIP
    for every mapping class $\phi$ in the splitting package of
    $\mathfrak{p}$.
  \end{enumerate}
\end{defin}

We observe the following properties. The first follows from
Lemma~\ref{lem:backtrack-implies-intersection} and our choice of
paths, while the second follows from the first since the number of
paths and endpoints is large enough.
\begin{itemize}
\item \textbf{(Not-PIP determines endpoint)} We have that for each
  mapping class $\Psi$, all of the $\mathfrak{p}\in M$ so that
  $(\phi,M)$ are not PIP share an endpoint.
  
\item \textbf{(Lots of Matches)} Let $(\Phi,(\eta_a,\eta_b))$ be in
  the splitting package of a model path. Then there is at most one
  model path endpoint $x$ on $\eta_a\cup \eta_b$, so that: for any
  other model path endpoint $x'$ on $\eta_a\cup \eta_b$, and at least
  $\frac 8 9$ of the model path endpoints $y$ on the opposite side we
  have $(\mathfrak{p}(x,y),\Phi)$ is PIP. 
\end{itemize}

The {directed} graph encoding the desired Markov process will be built in three stages. 
For the initial graph, let $V_0'$ be the set of
tuples
\[ (\Phi,(\eta_a,\eta_b)) \] of data associated to points in the
splitting packages in the model paths, and let $V_1'$ be the set of
model paths. We construct a (bipartite) directed graph $G'$ with vertex set
$V' = V'_0 \sqcup V'_1$ and oriented edges $E'$:
\begin{enumerate}[i)]
\item Edges with source $V_0'$ and range $V_1'$: we have an edge from $(\Phi,(\eta_a,\eta_b)) \in V_0'$ to
  $\mathfrak{p}\in V_1'$ if $\mathfrak{p}$ has one endpoint on
  $\eta_a$, one endpoint on $\eta_b$, and $(\Phi,\mathfrak{p})$ are
  a PIP.  
\item Edges with source $V_1'$ and range $V_0'$: we have an edge from $\mathfrak{p}$ to each of the $(\Phi,(\eta_a,\eta_b)) \in V_0$
  in the splitting package of $\mathfrak{p}$.  
\end{enumerate}

The next graph is obtained by removing edges. Namely, we obtain $E''$
from $E'$ by removing all the edges from $(\Phi,(\eta_a,\eta_b))$ to
$\mathfrak{p}(x_1,x_2)$ if one of the endpoints fails the condition
mentioned in (Lots of Matches). That is, we remove the edge if for
either of the endpoints $x_i$, for more than $\frac 1 9$ of the points
$y$ on the opposite face, $(\mathfrak{p}(x_i,y),\Phi)$ is not PIP. 
This removes the edges from at most one endpoint from each $\Phi$. Let $G''$ be the resulting graph.

The following is the analog of Corollary~\ref{cor:noback-to-thick} in this setting.
\begin{prop} There exists $\epsilon>0$  so that if $c=v_1,v_2,...$ is a (one-sided) infinite directed path in the above graph and $\phi_1,...$ be the mapping class coordinates then $\phi_1\gamma_0$, $\phi_1\phi_2\gamma_0$, .... converges to $\lambda \in \partial_\infty \mathcal{C}(S) $ and $\lambda$ is $\epsilon$-thick.
\end{prop}
\begin{proof} By the PIP condition $\gamma_0*\phi_1\gamma_0*\phi_1\phi_2*\gamma_0*....*\phi_1...\phi_k\gamma_0$ is a quasi-geodesic in the curve graph that is $\ell$-backtracking efficient. Let $X_0$ be a point in Teichm\"uller space with systole $\gamma_0$. By
Lemma \ref{lem:coboundedness-criterion}, the corresponding Teichm\"uller geodesic segments from $X_0$ to $\phi_1....\phi_kX_0$ are all $\epsilon$-thick and thus, so is the limiting Teichm\"uller geodesic ray, which is a ray from a base point in Teichm\"uller space to $\lambda$. By the definition of $\epsilon$-thick lamination, this means $\lambda$ is $\epsilon$-thick.
\end{proof}

The following is the core technical result which will be used in the
proof of path-connectivity of the limit set of the Markov process.

\begin{prop}\label{prop:graph connected}
Given 
\begin{enumerate}[i)]
\item a sequence $(\phi_0,(\eta^{(0)}_a,\eta^{(0)}_b)),....,(\phi_r,(\eta^{(r)}_a,\eta_b^{(r)})) \in {V''}$ so that 
\begin{itemize}
\item $\theta_i:=\phi_{i+1}^{-1} \phi_i\in S_{\eta_b}^{(i)}$ for all $i$
\item $\eta_a^{(i+1)}=\theta_i\eta_b^{(i)}$ for all $0\leq i<r$ 
\end{itemize}
\item $(\psi_s,(\zeta_a^{(s)},\zeta_b^{(s)}))$ is the endpoint of an edge starting in $ (\phi_0,(\eta^{(0)}_a,\eta^{(0)}_b))$
\item $(\psi_e,(\zeta_a^{(e)},\zeta_b^{(e)}))$ is the endpoint of an edge starting in $(\phi_r,(\eta_a^{(r)},\eta_b^{(r)}))$
\end{enumerate}
there exist
\begin{enumerate}
\item paths $\mathfrak{p}_0,...,\mathfrak{p}_r \in M$ 
\item\label{conc:descendant} there is an edge with source $(\phi_i, (\eta_a^{(i)},\eta_b^{(i)}))$ and range $\mathfrak{p}_i$. 
\item $\mathfrak{p}_i=\mathfrak{p}(x_i,y_i)$ where $x_i \in \theta_{i-1}S_{\theta_{i-1}\eta_b^{(i-1)}} \subset \eta_a^{(i)}$ and $y_i \in S_{\eta_b^{(i)}}$
\item $x_i=\theta_{i-1}y_{i-1}$. 
\item there is a path, $\mathfrak{p}_s$, so that 
\begin{itemize}
\item There is an edge with source $(\phi_0,(\eta_a^{(0)},\eta_b^{(0)}))$ and range $\mathfrak{p}_s$.
 \item $(\psi_s,(\zeta_a^{(s)},\zeta_b^{(s)}))$ is a splitting package for $\mathfrak{p}_s$.
\item $\mathfrak{p}_s$ and $\mathfrak{p}_0$ share an endpoint. 
\end{itemize}
\item Similarly,  there is a path, $\mathfrak{p}_e$ so that 
\begin{itemize}
\item There is an edge with source $(\phi_r,(\eta_a^{(r)},\eta_b^{(r)}))$ and range $\mathfrak{p}_e$
\item $(\psi_e,(\zeta_a^{(e)},\zeta_b^{(e)}))$ is a splitting package for $\mathfrak{p}_e$.
\item $\mathfrak{p}_e$ and $\mathfrak{p}_r$ share an endpoint. 
\end{itemize}
\end{enumerate}
\end{prop}

\begin{proof} 
  We begin choosing $\mathfrak{p}_s$, $\mathfrak{p}_e$ arbitrary satisfying the first and second bullet points of (5) and (6) respectively. This is possible because otherwise $(\psi_s,(\zeta_a^{(s)},\zeta_b^{(s)}))$ would not be the range of an edge with source
 $(\phi_0,(\eta_a^{(0)},\eta_b^{(0)}))$. Similarly for $(\psi_e,(\zeta_a^{(e)},\zeta_b^{(e)}))$. 

We now consider the endpoint of $\mathfrak{p}_s$ on $\eta_b^{(0)}=\theta_0^{-1}\eta_a^{(1)}$, which we denote by $y\in Y_{\eta_b^{(0)}}$ for now. 
If $\theta_0y$ satisfies the condition in (Lots of Matches) for $\phi_1$ we let $\mathfrak{p}_0=\mathfrak{p}_s$. 
Otherwise, let $x \in \eta_a^{(0)}$ be the endpoint of $\mathfrak{p}_s$ on $\eta_a^{(0)}$. 
We choose $\mathfrak{p}_0$ to be a path from $x$ to $y_0$ where $\theta_0y_0$ satisfies the condition in (Lots of Matches) for $\phi_1$. 
This exists by the fact that there is at most one $x' \in \eta_a^{(1)}$ that fails this condition for $\phi_1$. 

For $1\leq i<r-1$ we recursively choose $\mathfrak{p}_{i+1}$ given $\mathfrak{p}_i$ so that it satisfies \eqref{conc:descendant} in the conclusion of the proposition and moreover the endpoint of $\mathfrak{p}_{i+1}$ in $Y_{\eta_b^{(i+1)}}$, which we denote $y$ has that $\theta_{i+1}y$ satisfies the condition in (Lots of Matches) for $\phi_{i+2}$. We can do this because there is only one point in $Y_{\eta_{b}^{(i+1)}}$ which doesn't and we recursively have that the given endpoint for $\mathfrak{p}_{i+1}$ on $\eta_a^{(i+1)}$ satisfies the condition in (Lots of Matches) for $\phi_{i+1}$ so we can avoid the at most one forbidden endpoint on $\eta_b^{(i+1)}$. We now choose $\mathfrak{p}_{r-1}$ to satisfy the additional condition that its endpoint $y \in \eta_b^{(r-1)}$ has that $\mathfrak{p}(\theta_ry,y')$ is a child of $\phi_r$ in $E''$ where $y' \in Y_{\eta_b^{(r)}}$ is the endpoint of $\mathfrak{p}_e$ on $\eta_b^{(r)}$. We now choose $\mathfrak{p}_r$ to be $\mathfrak{p}_r:=\mathfrak{p}(\theta_ry,y')$.  
\end{proof}

Obtaining the splitting packages from the paths in the conclusion of  Proposition \ref{prop:graph connected}, we get the following corollary:
\begin{cor}\label{cor:graph connected} Under i)-iii) of the previous proposition we obtain:
\begin{enumerate}[(a)]
\item A sequence $(\psi_0,(\zeta_a^{(0)},\zeta_b^{(0)})),...,(\psi_q,(\zeta_a^{(q)},\zeta_b^{(q)})$
\item $\psi_{j+1}^{-1}\psi_j \in S_{\zeta_b^{(j)}}$ for all $j$.
\item\label{conc:LtoR} $\zeta_a^{(j+1)}=\psi_{j+1}^{-1}\psi_j\zeta_b^{(j)}$
\item $(\psi_s,(\zeta_a^{(s)},\zeta_b^{(s)})), (\psi_e,(\zeta_a^{(e)},\zeta_b^{(e)}))\in \{(\psi_j,(\zeta_a^{(j)},\zeta_b^{(j)}))\}_{j=0}^q$
\item\label{conc:monotonic} $d_1,...,d_r$ so that $\sum_{i=1}^r d_i=q+1$
and  $(\psi_j,\phi_i)$ are PIP for all $j\in [\sum_{\ell=1}^{i-1}d_\ell,\sum_{\ell=1}^id_\ell)$ 
\end{enumerate}
\end{cor}

The final graph $G=(V, E)$ is now obtained by  letting $V=V_0'$ and $E$ be all length two paths in $G''$ starting from a vertex in $V_0$'. 

We now define the \emph{limit set (based at $v_\ast$}) of the graph as the
set of all accumulation points of sequences of the form
\[ \cdots \psi_k \cdots \psi_1 \alpha_0. \] where for each $i$, there
are vertices in the graph $G$ containing $\phi_i, \phi_{i+1}$ in their
first coordinate which are joined by an oriented edge, and $\psi_1$ is
such for the vertex $v_\ast$.

The following is an immediate corollary of Proposition 6.2:
\begin{cor}
  Every point of the limit set (based at $v_\ast$) of the graph $G$ is a thick lamination.
\end{cor}

The rest of the section is concerned with the proof of the following
statement.
\begin{thm}\label{thm:markov path}
  The limit set (based at $v_\ast$) of the graph $G$ is path-connected.
\end{thm}
The strategy of the proof is similar to the construction employed in
Section~\ref{sec:thick-paths}, with paths in $G$ taking the place of
mapping class group sequences. The proof has two steps: Proposition \ref{prop:axiomatic}, which establishes some features of our construction and then to proof of Theorem \ref{thm:markov path} shows how these features imply path connectivity. 

To be more precise, suppose that we are given two (one-sided)-infinite paths
$c_0, c_1$ in the graph $G$, both starting at the vertex $v_\ast$. 
Inductively applying Proposition~\ref{prop:graph
  connected} and Corollary \ref{cor:graph connected} we obtain a ``fan of paths'' consisting of
vertices $v^{(j)}_k, k=0,\ldots, m_j$ so that:
\begin{prop}\label{prop:axiomatic}
Given a pair of directed, infinite paths in $(G,V)$, $c_0, c_1$ which 
emanate from the same vertex $v_0$ we have the following:
\begin{description}
\item[(Endpoints)] The ``outermost'' paths in the fan are the two given
  paths:
  \[ c_0({n}) = v^{(n)}_0, \quad c_1({n}) = v^{(n)}_{m_{n}}. \]
\item[(Refinement)] There is a monotonic decomposition
  \[ [1,m_{n+1}] = I_1 \cup \cdots \cup I_{m_n} \]
  into subintervals, so that for all $i \in I_j$ the vertices
  \[ v^{(n)}_j \quad\mbox{ and }\quad v^{(n+1)}_i \]
  are joined by an edge in $G$.

\item[(Left-to-right)] Denoting
  \[ v^{(n)}_i = (\phi^{(n)}_i, (\sigma^{(n)}_i, \eta^{(n)}_i)), \]
  we have that
  \[ \theta^{(n)}_i:=\left(\phi^{(n)}_{i+1}\right)^{-1} \phi^{(n)}_i\in S_{\eta^{(n)}_i} \]
  for all $i$ and 
  \[ \sigma^{(n)}_{i+1}=\theta^{(n)}_i\eta^{(n)}_{i}. \]
  
\item[(Paths)] There are paths
  \[ \mathfrak{p}^{(n+1)} = \mathfrak{p}^{(n+1)}_1 \ast \cdots \ast \mathfrak{p}^{(n+1)}_{m_n} \] obtained as concatenations of model paths,
  so that each mapping class $\phi$ appearing as the first component
  of a vertex $v^{(n+1)}_i, i \in I_j$ appears in a splitting package of
  $\mathfrak{p}^{(n)}_j$ and vice versa.
  
\end{description}
\end{prop}
\begin{proof}
This proceeds by induction. The base case is trivial being the empty path. 

Now we prove the inductive step, assuming we have the previous conclusions for $n=k$ and establish it { for}  $n=k+1$. We invoke Proposition \ref{prop:graph connected} and Corollary \ref{cor:graph connected} for the data from 
(Left-to-right) with $(\psi_s,(\zeta_a^{(s)},\zeta_b^{(s)}))=c_0(k+1)$ and $(\psi_e,(\zeta_a^{(e)},\zeta_b^{(e)}))=c_1(k+1)$. We obtain (Refinement) by  \eqref{conc:monotonic}. We obtain (Left-to-right) by  \eqref{conc:LtoR}. We obtain (Paths) by Proposition \ref{prop:graph connected}.
\end{proof}
\begin{proof}[Proof of Theorem \ref{thm:markov path}]
Now, for any vertex $v^{(n)}_i$ appearing in the fan, there is a unique chain
of ``ancestors''
\[ v^{(0)}_0, v^{(1)}_{i_1}, \ldots, v^{(n-1)}_{i_{n-1}}, v^{(n)}_i \]
so that $v^{(j)}_{i_j} \in I^{(j-1)}_{i_{j-1}}$. We define the
corresponding mapping class to be the product of the corresponding first entries:
\[ \Psi^{(n)}_i = \phi^{(1)}_{i_1} \circ \cdots \circ
  \phi^{(n-1)}_{i_{n-1}}. \] Observe that, by construction
(specifically, the existence of edges of $G$ between $v^{(j)}_{i_j}$ and $v^{(j+1)}_{i_{j+1}}$), the limits
\[ \lim_{n\to\infty} \Psi^{(n)}_{i_n}\alpha_0 \] exist and are
contained in the limit set (based at $v^{(0)}$).

Observe that by construction (specifically, the (Left-to-right) property) the
paths $\Psi^{(n)}_i\mathfrak{p}^{(n+1)}_i$ can be concatenated:
\[ \widehat{\mathfrak{p}}^{(n+1)} = \Psi^{(n)}_1\mathfrak{p}^{(n+1)}_1
  \ast \cdots \ast \Psi^{(n)}_{m_n}\mathfrak{p}^{(n+1)}_{m_n}. \]
 By the (Paths) property, we can parametrize the $\widehat{\mathfrak{p}}^{(n+1)}$ in such a way that if $t$ is a time so that 
  \[ \widehat{\mathfrak{p}}^{(n)}(t) = \Psi^{(n-1)}_k\mathfrak{p}^{(n)}_k(t_i) \]
  is a point in the $k$--th concatenated subpath, then
  \[ \widehat{\mathfrak{p}}^{(n+1)}(t) =
    \Psi^{(n)}_r\mathfrak{p}^{(n+1)}_r(t_r) \] is obtained by
  splitting that subpath further, i.e.
  $\Psi^{(n)}_r = \Psi^{(n-1)}_k\theta$, where $\theta$ comes from the
  splitting pacakge of the path $\mathfrak{p}^{(n)}_k$, and the corresponding vertices $v_k^{(n)}, v_r^{(n+1)}$ are joined by an edge.

  Now, the PIP condition guaranteed by the edges of $G$ implies that in this parametrization for any $R$ there exists $n_0$ so that for any $n,n'>n_0$ we have 
\begin{equation}\label{eq:further} \left( \widehat{\mathfrak{p}}^{(n)}(t),
      \widehat{\mathfrak{p}}^{(n')}(t) \right)>R
      \end{equation} for all $t \in [0,1]$. Indeed, as in the proof of Corollary \ref{cor:close}, there exists a uniform $C>0$, $D$ so that 
  \begin{itemize}
  \item $d_{\mathcal{C}}(\Psi_j^{(n)}\alpha_0,\alpha_0)>Cn$
  \item For any $\lambda$ in our limit set  and sequence $j_n$ so that $\lambda=\lim \Psi_{j_n}^{(n)}\alpha_0$ we have that the distance from the geodesic from $\alpha_0$ to $\lambda$ and $\Psi_{j_k}^{(k)}\alpha_0$ is at most $D$ for all $k$. 
  \end{itemize}
  
  Moreover, 
  since our parametrization respects the
  refinement of paths as in (Paths), for any constant $B > 0$, and any
  $t$ there is an interval $I$ containing $t$, and a constant $N_0$ so
  that
  \begin{equation} \label{eq:close on piece} \left( \widehat{\mathfrak{p}}^{(n)}(t),
      \widehat{\mathfrak{p}}^{(n')}(s) \right) > B 
      \end{equation} for any
  $n, n' > N_0, s \in I$. 
  Indeed, 
  we can choose $I$ small enough so that $s,s' \in I$ implies $\mathfrak{p}^{(n)}(s)$ and $\mathfrak{p}^{(n)}(s')$ come from adjacent intervals in the left-right order. Since $\max\{d_{\mathcal{C}}(\psi \beta,\beta):\psi \in S \text{ and }\beta \in \mathcal{C}(S)\}<\infty$, the first bullet point above this gives that for any $R'$ if $n$ is large enough (depending on $R'$),
\begin{equation}\label{eq:once far} \left( \widehat{\mathfrak{p}}^{(n)}(t),
      \widehat{\mathfrak{p}}^{(n)}(s) \right) >R'. 
      \end{equation}
      Combining this with \eqref{eq:further} we have \eqref{eq:close on piece}.

  
  Hence, the $\widehat{\mathfrak{p}}^{(n+1)} $ converge (in the Gromov boundary of
  the curve graph) to a continuous limit path $\widehat{\mathfrak{p}}^{(\infty)}$. 

  By the (Paths) property, every point of
  $\widehat{\mathfrak{p}}^{(\infty)}$ is contained in the limit set
  (based at $v^{(0)}$), since initial parts of splitting sequences
  toward a point on $\widehat{\mathfrak{p}}^{(\infty)}$ agree with
  splitting sequences toward points on $\widehat{\mathfrak{p}}^{(n)}$
  for $n$ large enough.

  By the (Endpoints) property, $\widehat{\mathfrak{p}}^{(\infty)}$
  connects the limit points defined by $c_1, c_2$. This proves the theorem.
\end{proof}

\begin{proof}[Proof of Theorem \ref{thm:markov}] 
We construct $H=\{h_1,...,h_k\}\subset MCG$ and a finite set of ordered pairs $Q\subset \{1,...,k\}^2$ so that if one considers 
$$\hat{H}=\{h_{n_1}...h_{n_j}:(n_{i},n_{i+1})\in Q \text{ for all }i<j\}$$ then the limit set of orbit map in the boundary of the curve graph $\hat{H}\gamma_0$ is path connected and contained in the cobounded laminations. 

Let $G=(V,E)$ be the graph we constructed. Let $H_0=\{\Phi:\exists (\Phi,(\eta_a,\eta_b)) \in V\}$. Choose a Pseudo-Anosov, $\Psi$, so that  $\Psi^n$ that has less that $l$ backtracking with respect to all $\Phi \in H_0$ for all  $n \in \mathbb{Z}$. For each vertex in $v\in V_0$ choose $N_v\in \mathbb{N}$. For any edge $e$ in $G$ let $A_e=\Psi^{N_v}\Phi\Psi^{-N_{v'}}$ where the source of $e$ is $v=(\Phi,(\eta_a,\eta_b))$ and the range of $e$ is $v'$. If 
$$D:=\min_{v\neq v'} \, \max\, \{\frac{N_v}{N_{v'}}, \frac{N_{v'}}{N_v}\}$$ is large enough, curve graph displacement shows $A_e\neq A_{e'}$ unless the unordered pairs of their range and source are the same. Because $\Psi$ and $\Psi^{-1}$ move in different directions along the axis, if $A_e=A_{e'}$, $e'$ cannot flip the range and source of $e$ (unless the range and source is the same in which case flipping the range and source doesn't change the edge).  Choose the $N_v$ so that $D$ is large enough and let $H$ be the set of $A_e$ where $e$ ranges over the edges of $G$.  
We choose our set of pairs $Q$ to allow concatenations of mapping classes, which are in bijection with the directed edges, iff the range of the edge corresponding to the first mapping class is the source of the edge corresponding to the second mapping class.

It is now clear by construction that the set of laminations which are obtained as limits
$$\lim_{j\to\infty} h_{n_1}...h_{n_j}\gamma_0, \quad (n_{i},n_{i+1})\in Q \text{ for all }i$$
is exactly $M_{v_0}\Lambda$, where $\Lambda$ is the limit set from
Theorem \ref{thm:markov path}. It remains to show that this set agrees
with the set of all accumulation points of $\hat{H}\gamma_0$ in the
boundary of the curve graph. So, let $\xi$ be a boundary point which
is a limit of curves $\Phi_k\gamma_0$ in $\hat{H}\gamma_0$. Each $\Phi_k$ is a product
$$\Phi_k\gamma_0 = h_{n^{(k)}_1}...h_{n^{(k)}_{r_k}}\gamma_0, \quad (n^{(k)}_{i},n^{(k)}_{i+1})\in Q \text{ for all }i<r_k.$$
Now, let $U$ be any neighbourhood of $\xi$ in
$\mathcal{C}\cup\partial_\infty\mathcal{C}$, i.e. the set of all
curves and boundary points whose Gromov product with $\xi$ is
sufficiently large. For large $k$, all curves $\Phi_k\gamma_0$ are
contained in $U$ and, given $R$, by passing to a subsequence we may assume that the
initial segment $n^{(k)}_i, i \leq R$ is the same for all such $k$. 

We now show that we may assume that $\Phi_{k+1}=\Phi_kh$. By the no-backtracking condition, and by choosing $R$ large enough, we can then see that
\emph{any} curve $\Phi\gamma_0$ 
$$\Phi\gamma_0 = h_{n_1}...h_{n_r}\gamma_0, \quad (n_{i},n_{i+1})\in Q \text{ for all }i<r\text{ and }n_i = n^{(k)}_i, i\leq R$$
is also contained in $U$. Now the claim follows since every finite
path in our graph starting in $v_0$ can be extended to an infinite
path, and therefore $U$ contains a point in $M_{v_0}\Lambda$. Thus, we have an infinite path that has its endpoint in $U$. Choosing a sequence of $U$, $U_1,U_2,...$ converging to $\xi$, for each $U_i$, we obtain an infinite sequence of elements of $H$, $h_{n_1^{(i)}},....$ so that the limit (fixing $i$ and taking $k\to \infty$) of $h_{n_1^{(i)}}...h_{n_k^{(i)}}\gamma_0$ is in $U_i$ and $(n_\ell^{(i)},n_{\ell+1}^{(i)}) \in Q$ for all $\ell$. By the no-backtracking condition and the fact that our paths make definite progress (by our choice of $B$), for each $j$ there exists an $r,N$ so that if $h_{m_1},... \in H$ is a sequence so that $(m_\ell,m_{\ell+1}) \in Q$ for all $\ell$ and $h_{m_\ell}=h_{n_\ell^{(i)}}$ for some $i\geq r$ and all $\ell \leq N$ then  the limit of $h_{m_1}...h_{m_k}\gamma_0$ is in $U_j$. From the sequences $(h_{n_\ell^{(i)}})_\ell$, we obtain a sequence $(h_{n_\ell^{(\infty)}})_\ell$ whose initial segment agrees with the initial segment of infinitely many of these. By our previous argument the limit of $h_{n_1^{(\infty)}}...h_{n_k^{(\infty)}}\gamma_0$ is in $U_j$ for all $j$ and thus its limit is $\xi$.

\end{proof}

\section{Appendix: A type of connected subset that can't be thick}

In this section we adapt an argument of Smillie for proving Masur's result that every translation surface contains a cylinder to rule out certain types of paths in the space of quadratic differentials, whose forward orbit under $g_t$ is uniformly thick. Much of this material is well known.

To state the result of this section, we need to define a metric on the space of unit area quadratic differentials in a fixed genus given by period coordinates: 
Let $\mathcal{Q}$ denote {the space of} quadratic differentials of a fixed topological type, $\mathcal{Q}(\vec{a})$ be a fixed stratum, $\mathcal{Q}_m$ be the corresponding marked quadratic differentials and $\mathcal{Q}_m(\vec{a})$ denote the corresponding marked stratum. 
 Let $g_t:=\begin{pmatrix}e^t&0\\0&e^{-t}\end{pmatrix}$. We follow a construction of Frankel \cite[Section 3.2]{Frankel} to define a metric $d_{per}(\cdot,\cdot)$ coming from period coordinates on $\mathcal{Q}$: Following \cite[Section 1]{MS} using a triangulation by saddle connections and a marking map, one can locally identify $\mathcal{Q}(\vec{a})$ with (a quotient of) relative cohomology on convex open subsets of $\mathcal{Q}(\vec{a})$. Choosing a basis on integer relative homology one can define a norm on relative cohomology with $\mathbb{C}$ coefficients. Frankel uses this to define a path metric on all of $\mathcal{Q}$ coming from these charts, which he denotes $d_{Euc}$ and we will denote $d_{per}$. Note, the charts are still convex but they are not necessarilly open in $\mathcal{Q}$. Frankel proves that for any compact subset $\mathcal{K} \subset \mathcal{Q}$,  and for any large enough upper bound $L$ the compact set has only finitely many such charts so that every saddle connection in the triangulations has length at most $L$ 
 \cite[Appendix A]{Frankel} 
 Note, because we are dealing with compact subsets of all of moduli space $\mathcal{Q}$, and not its individual strata, $\mathcal{Q}(\vec{a})$ there is no uniform lower bound on the length of saddle connections (and the dimensions of the relative cohomology spaces we are considering vary). 
 Throughout this section, we will deal with the preimange of compact subsets of moduli space in $\mathcal{Q}$ and strata $\mathcal{Q}(\vec{a})$ and obtain uniform statements over sets of this form.

We now can state our main theorem:

\begin{thm}\label{thm:main}  Let $\gamma:[0,1]\to \mathcal{Q}$ be a curve, $c\geq 0, \delta>0$ and $\mathcal{K}$ be the pre-image of a compact set in the moduli space of Riemann surfaces in $\mathcal{Q} \supset \mathcal{Q}(\vec{a})$. Assume 
 \begin{enumerate}
 \item (Uniform compact) $g_t\phi(s) \in \mathcal{K}$ for all $t\geq 0, $ $0\leq s\leq 1$. 
 \item (Uniform expansion) $d_{Per}(g_{t+r}\phi(a),g_{t+r}\phi(b))\geq  \delta e^{cr}d_{Per}(g_t\phi(a),g_t\phi(b))$ for all $0\leq a,b\leq 1$, $0\leq t$.
 \end{enumerate}
 Then $c\leq 1$. 
\end{thm}

The proof uses the following result, whose proof we defer to a later subset: 

\begin{prop}\label{prop:unif sc bound} For any compact subset of moduli space $\mathcal{K} \subset \mathcal{Q}$ and stratum $\mathcal{Q}(\vec{a})$
  and $u, r, c,\delta>0$ there exists an $L_{r,c,\mathcal{K},\vec{a}}:=L$ so that if $\phi:[0,1] \to \mathcal{Q}(\vec{a})$ is a path so that 
\begin{enumerate}
\item[i)] $g_t\phi(s)$ projects to $\mathcal{K}$ for all $0\leq t\leq L$ and $0\leq s\leq 1$
\item[ii)] $d_{per}(\phi(0),\phi(1))\geq r$.
\item[iii)] \label{cond:grow} $d_{per}(g_t\phi(a),g_t\phi(b))>
\min\{e^{ct}d_{per}(\phi(a),\phi(b)),u\} $
\end{enumerate}
{ then there exists $t \in [0,1]$ so that $\phi(t)$} has a vertical saddle connection with length in $[\frac 1 L,L]$. 
\end{prop}

It also uses the following consequence of the definition of the metric and that it is defined by only finitely many norms on any compact subset of moduli space: 

 \begin{lem} \label{lem:per growth geo}There exists $u>0$ so that the following is true. Suppose that $x, y \in \mathcal{Q}(\vec{a})$ are 
   connected by a path $\mathfrak{p} \subset \mathcal{Q}(\vec{a})$ and
   every point on $\mathfrak{p}$ 
   has injectivity radius at least $2\epsilon$. Further assume that $\gamma$ is a geodesic on $x$ and 
 $$\int_{\gamma}|dz_x|<\epsilon$$  for all $z \in \mathfrak{p}$.  Then we have the length of $\gamma$ on $y$ is at most $u$ 
 times the distance from $x$ to $y$ in period coordinates plus the length of $\gamma$ on $x$.  
 \end{lem} 
Note, we remove this multiplicative constant, $u$, to slightly lighten the notation in following proof.

\subsection{Proof of Theorem \ref{thm:main}} 
We choose a subsegment of $[a,b]\subset [0,1]$ so that $\phi([a,b])\subset \mathcal{Q}(\vec{a})$ for $\vec{a}$ fixed, which inherits the uniform compactness and uniform expansion properties. Becuase we are in a single stratum, we can choose a continuous one parameter family of homeomorphism $G_{t}:\phi(a) \to \phi(t)$ so that $G_a=Id$ and $G_t$ sends singularities to singularities for $t \in [a,b]$ and use these to identify singularities and separatrices. For example, we can locally do this on a patch with a fixed triangulation by the unique map that affinely sends triangles to triangles (see e.g. \cite[Section~2.4]{BSW}).

Thus we can (and will) identify one-chains and separatrices on different $\phi(s)$ for $s \in [a,b]$. 
Because for the preimage of each compact set of moduli space there is a uniform upper bound on the number of saddle connections of length less than 1, it suffices to prove the following: 

 We proceed iteratively, producing sets $\{\xi_1,...\xi_{k+1}\}$ of relative homotopy classes and surfaces $g_{t_1}\phi(s_1),...g_{t_{k+1}}\phi(s_{k+1})$ so that 
 \begin{enumerate}
 \item $\xi_j$ is realized as a saddle connection on $g_{t_j}\phi(s_j)$ (and so on $\phi(s_j))$.
 
 Note we do not assume that it is realizable as a saddle connection on $\phi(s_i)$ for $i \neq j$. 
 \item Letting $\gamma_{i,j}$ denote the geodesic representative of $\xi_j$ on $g_{t_i}\phi(s_i)$ then the length of $\gamma_{i,j}<\epsilon$ for all $j\leq i$. 
 \item $\sum_{j=1}^{k}\int_{\gamma_{k+1,j}}dz_{g_{t_{k+1}}\phi(s_{k+1})}<hol_{g_{t_{k+1}}\phi(s_{k+1})}(\xi_{k+1})$
 \item the injectivity radius of every point on every surface in $\{g_t\phi(s):t\in [0,\infty), s \in [0,1]\}$ is at least $2\epsilon$.
 \end{enumerate}
 
 Observe that if there is a $p_1,...p_r$ are singularities and there is a chain $\xi_{i_1},...\xi_{i_{r-1}}$ so that $\xi_{i_j}$ connects $p_j$ to $p_{j+1}$ and $\xi_{k+1}$ connect $p_r$ and $p_1$ then we obtain a contradiction. Indeed, we consider the path $\cup \gamma_{k+1,i_j} \cup \xi_{k+1}$ from $p_1$ to $p_1$ and observe $\int_{\cup \gamma_{k+1,i_j} \cup \xi_{k+1}}dz_{g_{t_i}\phi(s_i)}\neq 0$ so it is not homologically (and so not homotopically) trivial. Also the sum of the absolute values of the length is less than $2\epsilon$. This contradicts {(4)}.

 \begin{proof}[Completion of proof of Theorem \ref{thm:main}] 
 We choose $L$ in Proposition \ref{prop:unif sc bound} for some $r$. 
 The proof proceeds from induction as outlined above.   For the base case by Proposition \ref{prop:unif sc bound} there exists $a\leq s\leq b$ so that $\phi(s)$ has a vertical saddle connection, say of length $\ell$.  Thus $g_t\phi(s)$ has a vertical saddle connection of length $e^{-t}\ell$. If $t=\log(\frac{\ell}{\epsilon})$, assumption $(1)-(3)$ are obviously satisfied (since $k=0$). If $\epsilon$ is small enough, by our compactness assumption, inductive assumption (4) is satisfied. We now assume that the inductive assumption holds for $n$ geodesics (and all $\epsilon>0$)  
 and we want to produce an additional short saddle connection and a given $\epsilon>0$. 
 We apply the inductive hypothesis for obtain $g_t\phi(s)$ with $n$  geodesics 
 of length $\epsilon_0$, where $\epsilon_0$ will be determined later. Let $\gamma'$ be the connected component of the $\epsilon_0$ neighborhood of $g_t\phi(s)$. We now consider 
 $$g_{\log(r)-\frac 1 c \log(\epsilon_0)}\gamma':=\gamma''$$ and observe that by our uniform expansion assumption, its length in period coordinates is at least $r$. Now by Lemma \ref{lem:per growth geo} on every surface on $\gamma'$ our $n$ geodesics have length at most 
 $$2\epsilon_0.$$ Thus by estimating the expansion of $g_{\log(r)-\frac 1 c \log(\epsilon_0)}$ trivially, on every surface on $\gamma''$ their length is at most 
 $$r 2\epsilon_0^{\frac{c-1}c}.$$ By Proposition \ref{prop:unif sc bound} there is a surface $z$ on $\gamma''$ with a vertical saddle connection, $\xi_{n+1}$, of  length at least $\frac 1 L$ and at most $L$. 
We now consider 
 $$y:=g_{\log(\frac{L}{\epsilon})}z=g_{t+\log(r)-\frac 1 c \log(\epsilon_0)+\log(\frac{L}{\epsilon})}\phi(\sigma)$$ for some $\sigma \in [0,1]$ and observe that on $z$, $\xi_{n+1}$  has length at most  $\epsilon$. Our $n$ geodesics have length at most $\frac {L} {\epsilon} r2\epsilon_0^{\frac{c-1}c}$ on $y$. 
 We now state our assumption on $\epsilon_0$. We assume that 
 $$n  \frac L {\epsilon} r2\epsilon_0^{\frac{c-1}c}<\frac 1 2 \frac{\epsilon}{L}\min\{u, 1\}.$$ Since our saddle connection had length at most $\frac 1 L$, we have assumption (3) of the inductive assumptions we outlined before the proof. Because our saddle connection $\xi_{n+1}$ had length at most $L$ on $z$ it has length at most $\epsilon$ on $y$. So we satisfy assumption (1) and (2). If $\epsilon$ is small enough, (4) is satisfied by our uniform compactness assumption. 
 \end{proof}

 \subsection{Proof of Proposition \ref{prop:unif sc bound}}
 
 The proof uses two results the first is a uniform unique ergodicity statement that follows from Masur's criterion and our compactness assumptions:
 
 \begin{prop}\label{prop:intersection right} Let $q \in \mathcal{Q}$ so that $g_tq\in \mathcal{K}$ for all $t\geq 0$. For all $\epsilon>0$ there exists $L_\epsilon:=L$ so that for each  initial segment of a vertical separatrix of length $L$, $v$ and all pairs of saddle connections $\gamma,\zeta$ with transverse measure at least $\epsilon$ and length at most $\frac 1 \epsilon$ we have $$\bigg|\frac{i(v,\gamma)}{i(v,\zeta)}-\frac{\int_{\gamma}dx_q}{\int_{\zeta}dx_q}\bigg|<\epsilon^2.$$
\end{prop}

\begin{proof}
Assume not. For a fixed $\epsilon>0$ we obtain a convergent sequence of quadratic differentials $M_i$, saddle connections on $M_i$, $\gamma_i,\gamma_i'$ with transverse measure at least $\epsilon$ and legth at most $\frac 1 \epsilon$ and  separatrices $w_i$ on $M_i$ so that the initial segment of $w_i$ of length $L_i$, $v_i$ has 
\begin{itemize}
\item $|\frac{i(v_i,\gamma_i)}{i(v_i,\gamma'_i)}-\frac{\int_{\gamma_i} dx}{\int_{\gamma_i'}dx}|>\epsilon^2$
\item $L_i \to \infty$
\end{itemize}
Passing to the limit, we obtain 
\begin{itemize}
\item a limiting $M \in \mathcal{Q}$, with $g_tM \in \mathcal{K}$ for all $t\geq 0$ and so is uniquely ergodic
\item the limit of these separatrices converges to a union of separatrices $S$ that meet at cone points and have pairwise angle $\pi$ so that the last separatrix is half-infinite (otherwise there is a closed curve in the union which contradicts unique ergodicity)
\item  and a pair of finite unions of saddle connections $\gamma,\gamma'$ 
 with transverse measure at least $\epsilon$ and length at most $\frac 1 \epsilon$. 
\end{itemize}
so that the intersections of the terminal separatrix in $S$ with
$\gamma$ is not appropriate.  To see this, since we cover
$\mathcal{K}\subset \mathcal{Q}$ with only finitely many coordinate
patches, we may assume the $M_i$ are all in a fixed coordinate patch
and since $\mathcal{K}$ is compact we may assume they converge to a
limit $M$ (not necessarily in this coordinate patch). The seperatrices
$\gamma_i$ converge (up to a further subsequence) to a seperatrix of
$M$. To see this, realise e.g. the $M_i$ as actual $1$--forms on a
fixed surface (as opposed to isotopy classes), and observe that the
seperatrices are paths where these $1$--forms evaluates as real
numbers. Compare \cite[Section 2]{LS} for details on realising by
actual diffeomorphisms. 

In this limit differential, the limit seperatrix now has the wrong intersections -- which contradicts unique ergodicity.  
\end{proof}

The second is about how our metric interacts with the geodesic flow:

\begin{lem}\label{lem:horiz change}Under the assumptions of Proposition \ref{prop:unif sc bound} for every $\delta>0$ there exists $\epsilon>0$ so that if 
$d_{per}(\phi(a),\phi(b))>\delta$ then there are $c,d$ between $a$ and $b$ so that the segment between $c$ and $d$ is sent by $\phi$ to the same coordinate patch and there is a pair of saddle connections on $\phi(c)$, $\phi(d)$, $\gamma,\zeta$ so that $\phi(a)$ and $\phi(b)$ both $\gamma$ and $\zeta$ have length between $\epsilon$ and $\frac 1 \epsilon$ and 
$$\bigg|\frac{\int_\gamma dx_{\phi(c)}}{\int_{\zeta}dx_{\phi(c)}}-\frac{\int_\gamma dx_{\phi(d)}}{\int_\zeta dx_{\phi(d)}}\bigg|>\epsilon.$$
\end{lem}
We again defer the proof of the above lemma to the next subsection,
and first explain how it implies the main result of this section.
\begin{proof}[Proof of Proposition \ref{prop:unif sc bound}]
Let $c,d,\gamma,\zeta$ be as in Lemma \ref{lem:horiz change}. Using $\phi$, we have a family of (piecewise) affine homeomorphism on the surfaces on the segment between $\phi(c)$ and $\phi(d)$ which we can use to identify the separatrices on $\phi(c)$ and $\phi(d)$ (and every surface on the segment in between). 
By Proposition \ref{prop:intersection right} and our assumption on $\phi$ there exists $L$ so that for any vertical separatrix that has length $L$ on both $\phi(c)$ and $\phi(d)$, have different ratios of intersections with $\gamma$ and $\zeta$. We now choose $k$ so that for any $s$ between $c$ and $d$, for any vertical segment of length $L$ that intersects $\gamma$ at least $k$ times must have length at least $L$. We now consider a segment of a vertical separatrix on $\phi(c)$, $\sigma_c$ that intersects $\gamma$ exactly $k$ times. If this is not a vertical saddle connection, using our homeomorphisms we can consider a corresponding segments of $\sigma_s$ on $\phi(s)$ for $s$ near $c$. This separatrix becomes a (vertical) saddle connection when an intersection is added or removed from $\gamma$. This has to occur between $\phi(c)$ and $\phi(d)$. Indeed, letting $\sigma_c(k)$ and $\sigma_d(k)$ be the segments of $\phi(c)$ and $\phi(d)$ that intersect $\gamma$ exactly $k$ times we have $i(\sigma_c(k),\zeta)\neq i(\sigma_d(k),\zeta)$. 
\end{proof}

\subsubsection{Proof of Lemma \ref{lem:horiz change}}

We consider $d_{per}$ restricted to $\mathcal{Q}(\vec{a})$ and its properties under $g_t$. Fix $\mathcal{K}$ a compact subset of $\mathcal{Q}$. 
Then we claim that there exists $\tilde{\epsilon}>0$ so that if $x,y \in \mathcal{Q}(\vec{a})$ and can be connected by a path in $Q(\vec{a})$ of length at most $\tilde{\epsilon}$ then there is a path in $\mathcal{Q}(\vec{a})$ connecting $x,y$ with length at most $(1+\delta)d_{per}(x,y)$ for any $\delta>0$.
This follows since, although $\mathcal{Q}(\vec{a})$ is not necessarily isometrically embedded in $\mathcal{Q}$, there are only finitely many convex, polyhedral, Masur-Smillie patches, (with maximal length of any saddle connection at most $L$) that intersect $\mathcal{K}$.

We say $x,y \in \mathcal{Q}(\vec{a})$ are in the same stable manifold if  there is a path in $\mathcal{Q}(\vec{a})$ so that the horizontal components of the holonomies of one-chains on each of point in the path is the same. If $\mathcal{P}$ is a Masur-Smillie patch, we say $x,y \in \mathcal{P}$ are in the same local stable manifold if  there is a path in $\mathcal{P}$ so that the horizontal components of the holonomies of one-chains on each of point in the path is the same.

\begin{lem}\label{lem:stable} There exists $\epsilon>0,D$ depending only on the compact part of moduli space, $\mathcal{K}$ so that if $x,y \in \mathcal{Q}(\vec{a})$ are in the same stable manifold and there is a path connecting them in the stable manifold of length $d_{per}(x,y)<\epsilon$ for all $t>0$ and $g_tx,g_ty \in \mathcal{K}$  then $d_{per}(g_tx,g_ty)<D\epsilon$ for all $t>0$. 
\end{lem}
\begin{proof}
  By Frankel we can cover the path by convex charts where the triangulation (but not necessarilly the norm) is constant.
  Because these patches are convex, and there are only finitely many norms in the definition of $d_{per}$ on $\mathcal{K} \subset \mathcal{Q}$ and so they are uniformly comparable. It suffices to prove the lemma for a straight line segment $\mathcal{\ell}$ in the stable manifold in a single convex patch $\mathcal{P}_s$, which is pushed forward by $g_t$ to a straightline segment in the stable manifold in a  pushed  single patch $\mathcal{P}_e$. 

We now compare two computations in $\mathbb{C}^k$ with two different norms: Let $\|\cdot \|_{s}$ be the norm given by the chart coming from Masur-Smillie on $\mathcal{P}_s$ (which is comparable by uniform multiplicative constants to the Frankel norms on this set) and $\|\cdot \|_e$ be the Masur-Smillie norm on $\mathcal{P}_e$. We need to compare the pushforward by $g_t$ of $\|\cdot \|_s$ with $\|\cdot \|_e$. This comparison is given by the integer matrix $A_t$ that changes $g_t\mathcal{B}_s$ to $\mathcal{B}_e$ 
where $\mathcal{B}_s$ and $\mathcal{B}_e$ are the bases of integer relative homology for $\mathcal{P}_s$ and $\mathcal{P}_e$ respectively. 
Because we are in a compact set in moduli space and there is an upper bound on the length of each element of $\mathcal{B}_s$ (maximized over the metrics given by points in $\mathcal{P}_s$) we have that there exists $C_{\mathcal{K}}$ so that every entry of $A_t$ is at most $C_{\mathcal{K}}e^t$. 
To show this, for $b \in \mathcal{B}_s$ we want to write $g_tb$ as a sum of elements of $\mathcal{B}_e$. We consider $g_tb \in g_tM \in \mathcal{P}_e$ where $M \in \mathcal{B}_s$. Using the triangulation of the surfaces in $\mathcal{P}_e$ we can write this (or any other) saddle connection as the sum of saddle connections in the triangulation. Because there is a uniform lower bound on the length of the shortest simple closed curve, there is a $N_{\mathcal{K}}$ so that the number of times a saddle connection in the triangulation can be used in this representation is at most $\frac{length(g_tb)}{N_{\mathcal{K}}}$. Now,  the coefficients in the representation of the saddle connections in the triangulation by relative homology classes is uniformly bounded from above. (By finiteness of patches.) The claim on the entries of $A_t$ follows because the length of the elements of $g_t\mathcal{B}$ on surfaces in $\mathcal{P}_e$ are at most $e^tL_{\mathcal{K}}$, where $L_{\mathcal{K}}$ is the longest length of any element in any basis over $\mathcal{K}$. 

The lemma is completed because the length of $g_t  \ell$ with respect to the $g_t$ pushforward of $\|\cdot\|_s$ is $e^{-t}$ times the length of $\ell$. 
\end{proof}

Because there is a uniform upper bound on the length of longest saddle connection in the bases for relative homology that define our norms we have the following:

\begin{lem}\label{lem:central stable}There exists $D$ so that if $p,q$ are in the same Masur-Smillie patch and for each pair of one chains $\gamma,\zeta$ we have 
$$\frac{\int_\gamma dx_{p}}{\int_{\zeta}dx_{p}}=\frac{\int_{\gamma} dx_{q}}{\int_{\zeta} dx_{q}} $$ then $d_{per}(g_tp,g_tq)<Dd_{per}(p,q)$.  
\end{lem}

Given a  Masur-Smillie patch $\mathcal{P}$ we define the local central-stable manifold through $x \in \mathcal{P}$ to be 
$$\bigcup_{y \in \mathcal{S}(x)} \bigcup _{s_y<t<e_y} g_t y$$ where 
$s_y=\max\{t<0:g_ty \notin\mathcal{P}\}$ and $s_e=\min\{t>0:g_ty \notin\mathcal{P}\}$ and $\mathcal{S}(x)$ is the local stable manifold through $x$. 

\begin{cor}\label{cor:constant D}There exists $D$ so that for all $x$ and $y$ in the central-stable manifold through $x$ so that 
$$d_{per}(g_sx,g_sy)<D d_{per}(x,y)$$ for all $s>0$.
\end{cor}

\begin{proof}[Proof of Lemma \ref{lem:horiz change}] 
Let $\|\cdot\|_s$ be the norm on the Masur-Smillie patch, $\mathcal{P}$ and $C$ be so that for all $x,y \in \mathcal{P}$ we have 
$$\frac 1 C \|\Phi^{-1}(x)-\Phi^{-1}(y)\|_s\leq  d_{per}(x,y)\leq C \|\Phi^{-1}(x)-\Phi^{-1}(y)\|_s$$ where $\Phi$ is the map from relative cohomology to the Masur-Smillie patch giving our coordinates. Note, $\Phi$ could finite to one in which case we choose the elements of $\Phi^{-1}(x)$ and $\Phi^{-1}(y)$ that are closest in the metric coming from $\|\cdot \|_s$. We will continue this convention throughout.  Let $D$ be as in Corollary \ref{cor:constant D}. 
Now if 
$$\|\Phi^{-1}(\phi(c))-\Phi^{-1}(\phi(d))\|_s=\frac{r}{4DC}$$ then $$d_{per}\bigg(g_{\frac 1 c \log(4DC)}\phi(c),g_{\frac 1 c \log(4DC)}\phi(d)\bigg)\geq r .$$ 
Now if $x,y$ are in the same local central-stable manifold then $$d_{per}\bigg(g_{\frac 1 c \log(4DC)}\phi(c),g_{\frac 1 c \log(4DC)}\phi(d)\bigg)\leq \frac r{4}. $$
Because $g_{\frac 1 c \log(4DC)}$ is a fixed $d_{per}$-continuous map and $\mathcal{K}$ is compact, we have that $y$ is definite distance from the connected component of the central stable manifold through $x$ intersected with $\mathcal{P}$. By the definition of our metric we have the lemma.  
\end{proof}

  \subsection{Properness of assumptions}
 
 We now show that the growth condition on our path is proper. We construct a path $\phi:[0,1] \to \mathcal{Q}_m$, in a single stratum of quadratic differentials,\footnote{it will be in the stratum of abelian differentials $\mathcal{H}(2,2,1,1)$ and there is a fixed flat torus where every point on the path is a branched cover of this torus} so that 
 \begin{itemize}
  \item there exists a single compact set of moduli space $\mathcal{K}$ so that $g_t \phi(s) \in \mathcal{K}$ for all $t\geq 0$ and $s\in [0,1]$. 
 \item $d_{per}(g_t\phi(a),g_t\phi(b))\geq Ce^td_{per}(\phi(a),\phi(b))$ for all $a,b\in [0,1]$ and $t\geq 0$.
 \end{itemize}

 Consider $M$, 3-to-1 branched cover of the torus in the stratum of abelian differentials $\mathcal{H}(2)=\mathcal{Q}(4,+)$, so that Teichm\"uller geodesic on the surface (equivalently the torus) projects to a fixed compact part of moduli space. Glue two identical copies of this surface along a slit, so that the projections of the endpoints are $B$-badly approximable. Let $\phi$ move both endpoints horizontally to the left at unit speed. As in \cite[Section 3]{Paths}, the $B$-badly approximable property is preserved along $\phi$ and every surface is thick.
 
 Applying Teichm\"uller geodesic flow to $\phi$, one obtain $\phi_t$ which moves both singularities to the left at speed $e^t$. 
At each point one can choose bases of relative cohomology so that locally all but two elements are unaffected and two basis elements have their horizontal components increase with speed $e^t$ and the vertical components remain unchanged. Indeed, the unchanging basis vectors can be chosen from a basis of the original torus. Moreover, this can be done where each basis element is realized by saddle connections with length uniformly bounded from above. By the fact that any basis we choose in the definition has this property, the change seen from this choice is bilipschitz to $d_{per}$ and the claim follows. 

\bibliographystyle{alpha}
\bibliography{ref-markov}

\begin{thebibliography}{BSW22}

\bibitem[BH99]{BH}
Martin~R. Bridson and Andr{\'e} Haefliger.
\newblock {\em Metric spaces of non-positive curvature}, volume 319 of {\em
  Grundlehren Math. Wiss.}
\newblock Berlin: Springer, 1999.

\bibitem[BSW22]{BSW}
Matt Bainbridge, John Smillie, and Barak Weiss.
\newblock {\em Horocycle dynamics: new invariants and eigenform loci in the
  stratum {{\(\mathcal{H} (1,1)\)}}}, volume 1384 of {\em Mem. Am. Math. Soc.}
\newblock Providence, RI: American Mathematical Society (AMS), 2022.

\bibitem[CH24]{Paths}
Jon Chaika and Sebastian Hensel.
\newblock Path-connectivity of the set of uniquely ergodic and cobounded
  foliations.
\newblock {\em Geometriae Dedicata}, 218(6):109, 2024.

\bibitem[Cor17]{Cor}
Matthew Cordes.
\newblock Morse boundaries of proper geodesic metric spaces.
\newblock {\em Groups Geom. Dyn.}, 11(4):1281--1306, 2017.

\bibitem[CSZ]{ChCorSis}
Matthew Cordes, Alessandro Sisto, and Stefanie Zbinden.
\newblock Corrigendum to morse boundaries of proper geodesic metric spaces.

\bibitem[Fra19]{Frankel}
Ian Frankel.
\newblock A comparison of period coordinates and teichm\"uller distance, 2019.

\bibitem[Ham05]{Ham}
Ursula Hamenstaedt.
\newblock Word hyperbolic extensions of surface groups, 2005.

\bibitem[LS09]{LS}
Christopher~J. Leininger and Saul Schleimer.
\newblock Connectivity of the space of ending laminations.
\newblock {\em Duke Math. J.}, 150(3):533--575, 2009.

\bibitem[LS14]{LS-thick}
Christopher~J. Leininger and Saul Schleimer.
\newblock Hyperbolic spaces in {Teichm{\"u}ller} spaces.
\newblock {\em J. Eur. Math. Soc. (JEMS)}, 16(12):2669--2692, 2014.

\bibitem[Min96]{Minsky}
Yair~N. Minsky.
\newblock Quasi-projections in {Teichm{\"u}ller} space.
\newblock {\em J. Reine Angew. Math.}, 473:121--136, 1996.

\bibitem[Min05]{Min05}
Ashot Minasyan.
\newblock On residualizing homomorphisms preserving quasiconvexity.
\newblock {\em Commun. Algebra}, 33(7):2423--2463, 2005.

\bibitem[Mos03]{Mosher}
Lee Mosher.
\newblock Train track expansions of measured foliations.
\newblock Preprint, 2003.

\bibitem[MS91]{MS}
Howard Masur and John Smillie.
\newblock Hausdorff dimension of sets of nonergodic measured foliations.
\newblock {\em Ann. Math. (2)}, 134(3):455--543, 1991.

\bibitem[PH92]{PH}
R.~C. Penner and J.~L. Harer.
\newblock {\em Combinatorics of train tracks}, volume 125 of {\em Ann. Math.
  Stud.}
\newblock Princeton, NJ: Princeton University Press, 1992.

\bibitem[Raf14]{Rafi-hyperbolic}
Kasra Rafi.
\newblock Hyperbolicity in {Teichm{\"u}ller} space.
\newblock {\em Geom. Topol.}, 18(5):3025--3053, 2014.

\end{thebibliography}

\end{document}